\documentclass[12pt]{amsart}

\usepackage{tikz-cd}
\usepackage[left=25mm,right=25mm, top=25mm, bottom=25mm]{geometry}
\usepackage[english]{babel}
\usepackage{tikz-cd}
\usepackage[T1]{fontenc}
\usepackage[nice]{nicefrac}
\usepackage{amsmath}
\usepackage{amstext}
\usepackage{amsfonts}
\usepackage{amssymb}
\usepackage{amsthm}
\usepackage{mathtools}
\usepackage{dsfont}
\usepackage{color}
\usepackage{graphicx}
\usepackage[colorinlistoftodos]{todonotes}
\usepackage[colorlinks=true, allcolors=blue]{hyperref}
\usepackage{float}
\usepackage[colorinlistoftodos]{todonotes}
\usepackage{theoremref}
\usepackage{enumitem}
\usepackage{dirtytalk}
\usepackage{todonotes}
\usepackage{marginnote}

\usepackage{multicol}
\usepackage{titletoc}
\allowdisplaybreaks
\let\oldaddcontentsline\addcontentsline

\newcommand{\starttocentries}{\let\addcontentsline\oldaddcontentsline}
\newtheorem{theorem}{Theorem}[section]
\newtheorem{lemma}[theorem]{Lemma}
\newtheorem{prop}[theorem]{Proposition}
\newtheorem{cor}[theorem]{Corollary}
\newtheorem{lem}[theorem]{Lemma}

\newtheorem{question}[theorem]{Question}
\newtheorem*{cor*}{Corollary}
\newtheorem*{conjecture*}{Conjecture}
\newtheorem*{thm*}{Theorem}
\newtheorem*{lem*}{Lemma}

\newtheorem*{prop*}{Proposition}
\theoremstyle{definition}
\newtheorem{definition}[theorem]{Definition}

\newtheorem{example}[theorem]{Example}
\newtheorem*{defn*}{Definition}

\theoremstyle{remark}
\newtheorem{remark}[theorem]{Remark}

\newcommand{\notimplies}{%
  \mathrel{{\ooalign{\hidewidth$\not\phantom{=}$\hidewidth\cr$\implies$}}}}

\newcommand{\E}{\mathbb{E}}

\makeatletter \def\l@subsection{\@tocline{2}{0pt}{1pc}{5pc}{}} \def\l@subsection{\@tocline{2}{0pt}{2pc}{6pc}{}} \makeatother

\title{Invariant $C^*$-subalgebras of the reduced group $C^*$-algebra}
\author[Amrutam]{Tattwamasi Amrutam}
\address{Institute of Mathematics of the Polish Academy of Sciences, Ul. S'niadeckich 8\\ 00-656 Warszawa, Poland}
\email{tattwamasiamrutam@gmail.com}
\author[Jiang]{Yongle Jiang*}
\address{School of Mathematical Sciences, Dalian University of Technology, Dalian, 116024, China}
\email{yonglejiang@dlut.edu.cn}
\date{\today}
\thanks{*-Corresponding author}
\begin{document}
\begin{abstract}
Let $\Gamma$ be a countable discrete group. We say that $\Gamma$ has $C^*$-invariant subalgebra rigidity (ISR) property if every $\Gamma$-invariant $C^*$-subalgebra $\mathcal{A}\le C_r^*(\Gamma)$ is of the form $C_r^*(N)$ for some normal subgroup $N\triangleleft\Gamma$. We show that all torsion-free, non-amenable (acylindrically) hyperbolic groups, and a finite direct product of such groups have this property. We also prove that an infinite group $\Gamma$ has the C$^*$-ISR property only if $\Gamma$ is simple amenable or $C^*$-simple.  \end{abstract}
\maketitle
\tableofcontents
\section{Introduction}
Let $\Gamma$ be a discrete group. Many profound results have appeared in the past that determine the structure of an intermediate subalgebra associated with inclusions (see, e.g.,~\cite{izumi1998galois,cameron2016intermediate, cameron2019galois, Suz,rordam2023irreducible} among others). These results establish a \say{rigidity phenomenon} under certain assumptions. In this paper, we are interested in \textbf{unital} C$^*$-subalgebras $\mathcal{A}$ associated with the inclusion $\mathbb{C}\subset C_r^*(\Gamma)$; without the unital assumption, $C^*$-simplicity is clearly necessary for the rigidity properties we discuss.

It is interesting to ask which group properties are reflected at the group von Neumann algebra or reduced group $C^*$-algebra level. Recall that a group $\Gamma$ is said to have the Invariant Subalgebra Rigidity (ISR) property in the von Neumann setting if every $\Gamma$-invariant von Neumann subalgebra of $L(\Gamma)$ is of the form $L(N)$ for some normal subgroup $N \triangleleft \Gamma$~\cite{amrutam2023invariant}. Analogously, a normal subgroup $N\triangleleft\Gamma$ gives rise to a $\Gamma$-invariant $C^*$-subalgebra $C_r^*(N)\le C_r^*(\Gamma)$.
When we say invariance, we mean invariance under the conjugation action of the unitary elements coming from the group $\Gamma$. We say that $\Gamma$ is \say{rigid} in some sense if $\{C_r^*(N): N\triangleleft\Gamma\}$ exhaust the list of all invariant $C^*$-subalgebras of $C_r^*(\Gamma)$. Note that all invariant C$^*$-subalgebras considered in this paper are assumed to be unital.
\begin{definition}[$C^*$-ISR property]
Let $\Gamma$ be a discrete group. We say that $\Gamma$ has $C^*$-invariant subalgebra rigidity property ($C^*$-ISR property) if every invariant $C^*$-subalgebra $\mathcal{A}\le C_r^*(\Gamma)$ is of the form $C_r^*(N)$ for some normal subgroup $N\triangleleft\Gamma$.
\end{definition}
In this paper, we establish a new rigidity phenomenon by showing that there is a significant class of groups $\Gamma$ for which the only invariant $C^*$- subalgebras inside $C_r^*(\Gamma)$ come from the normal subgroups; in particular, they satisfy the $C^*$-ISR property.
\begin{theorem}
Let $\Gamma$ be a torsion-free non-amenable hyperbolic group. Then, $\Gamma$ has the $C^*$-ISR property.
\end{theorem}
Our methods also work for all torsion-free acylindrically hyperbolic groups (see Subsection~\ref{sub:torsionacylin}). We also show that a finite product of such groups satisfies the $C^*$-ISR property.
\begin{theorem}
Let $\Gamma=\Gamma_1\times\Gamma_2\times\cdots\times\Gamma_n$, where each $\Gamma_i$ is a torsion-free non-amenable hyperbolic group. Then, $\Gamma$ has the $C^*$-ISR property.
\end{theorem}
An immediate consequence of the $C^*$-ISR property is that for $C^*$-simple groups with this property, every invariant $C^*$-subalgebra is simple (see~\cite[Theorem~1.4]{breuillard2017c}). While it was shown in \cite[Theorem~1.3]{amrutam2020simplicity} that every invariant $C^*$-subalgebra $\mathcal{A}\le C_r^*(\Gamma)$ for a $C^*$-simple group is $\Gamma$-simple (i.e., $\mathcal{A}$ does not have any non-trivial $\Gamma$-invariant ideals), $C^*$-ISR rigidity allows us to conclude that such invariant ones are also simple.
We also obtain a dichotomy for infinite groups satisfying the $C^*$-ISR property in that they are either $C^*$-simple or simple amenable.
\begin{theorem}
\label{thm:conseISR}
Let $\Gamma$ be a countable infinite group that satisfies the $C^*$-ISR property. Then, either $\Gamma$ is $C^*$-simple or is a simple amenable group.   \end{theorem}
Much work has been done to study the invariant von Neumann subalgebras of the group von Neumann algebra. Motivated by the works of \cite{alekseev2019rigidity} and \cite{chifan2020rigidity}, Kalantar-Panagoupolos~\cite{kalantar2022invariant} showed that every $\Gamma$-invariant sub-algebra of $L(\Gamma)$ is of the form $L(N)$ for a normal subgroup $N\triangleleft\Gamma$, where $\Gamma$ is an irreducible lattice in the product of higher rank simple Lie groups. This motivated the authors to introduce \say{ISR-property} in \cite{amrutam2023invariant}  in the context of the group von Neumann algebra $L(\Gamma)$ for a general countable discrete group, and we established this property for all
 torsion-free hyperbolic groups $\Gamma$, and finite direct product of them. This was partially generalized in \cite{chifan2022invariant} for a broader class of groups, including non-amenable acylindrically hyperbolic groups with trivial amenable radicals.

 More recently, the second named author, along with Zhou~\cite{jiang2024example}, has exhibited an infinite amenable group with the ISR-property (also see \cite{dudko2024character}). The ISR property, however, does not imply the C$^*$-ISR property in general (see Example \ref{example: ISR does not imply C-ISR}). Irrespective, if an infinite amenable group satisfied the $C^*$-ISR property, then the ideal structure becomes much more transparent.
\begin{theorem}
\label{thm:consequenceofISR}
Let $\Gamma$ be a countable infinite amenable group with the C$^*$-ISR property. Then, the only non-trivial closed two-sided ideal in $C^*_r(\Gamma)$ is $J_{\text{Aug}}$, where $J_{\text{Aug}}$ is the augmentation ideal generated by $\{\lambda(g)-\lambda(e):~g\in\Gamma\}$. Moreover, $\mathbb{C}\Gamma$ has a unique C$^*$-completion.
\end{theorem}
A key point in these arguments to show ISR-property in the setting of $L(\Gamma)$ is that $L(\Gamma)$ is a finite von Neumann algebra. As such, every von Neumann subalgebra $\mathcal{M}\le L(\Gamma)$ is the image of a conditional expectation. This is certainly not the case for subalgebras $\mathcal{A}$ inside $C_r^*(\Gamma)$ for the structure is not so rich (see, for example, \cite[Example~1.2]{pitts2017structure} or Section~6 in there for a far more general setup). This may suggest working with \say{pseudo conditional expectations} instead. This notion has been recently introduced and studied by Pitts~\cite{pitts2017structure, pitts2021structure}. Given an inclusion of $\Gamma$-$C^*$-algebra $\mathcal{B}\subset\mathcal{C}$, a pseudo conditional expectation is a $\Gamma$-equivariant map $\phi:\mathcal{C}\to I_{\Gamma}(\mathcal{B})$ such that $\phi|_{\mathcal{B}}=\text{id}$. Such a map always exists by the $\Gamma$-injectivity of the $\Gamma$-injective envelope $I_{\Gamma}(\mathcal{B})$.

However, the $\Gamma$-injective envelope is a highly intractable object in general. Moreover, the action $\Gamma\curvearrowright I_{\Gamma}(\mathcal{A})$ is not the conjugation action. Therefore, the techniques used to show ISR-property for the group von Neumann algebra do not have an immediate modification. In order to counter this drawback, we resort to using the averaging technique to \say{kill off} certain elements. This averaging technique is reminiscent of the Powers' averaging principle established in \cite{haagerup2016new} for $C^*$-simple groups. We refer the readers to Subsection \ref{subsection: general strategy} and, in particular, Proposition \ref{prop:genstrategy} for a general description of the strategy.

\subsection*{Acknowledgments}
Y. J. is partially supported by the National Natural Science Foundation of China (Grant No. 12471118). We thank Adam Skalski, Hanfeng Li, and Bartosz Kwa\'sniewski for taking the time to read through our draft and for suggesting numerous changes that improved the readability of the paper. We also thank Hanfeng Li for showing us Proposition~\ref{prop:supoprtofsubalg}. We are also grateful to the anonymous referee for useful comments that greatly improved the exposition.

\section{Preliminaries}

\subsection{Reduced Group \texorpdfstring{$C^*$}{C*}-algebra}
\label{sec:preliminary}
We briefly recall the construction of the group C$^*$-algebra. Let $\ell^2(\Gamma)$ be the space of square summable $\mathbb{C}$-valued functions on $\Gamma$. There is a natural action $\Gamma\curvearrowright \ell^2(\Gamma)$ by left translation:
\[\lambda_g\xi(h):=\xi(g^{-1}h), \xi \in \ell^2(\Gamma), g,h \in \Gamma\]
The reduced group $C^*$-algebra $C_r^*(\Gamma)$ is generated as a $C^*$-algebra inside $\mathbb{B}(\ell^2(\Gamma)$), by the left regular representation $\lambda$ of $\Gamma$. The reduced group $C^*$-algebra $C_r^*(\Gamma)$ comes equipped with a canonical trace $\tau_0:C_r^*(\Gamma)\to\mathbb{C}$ defined by
\[\tau_0\left(\lambda_g\right)= \begin{cases}
0 & \mbox{if $g\ne e$}\\
1 & \mbox{if $g=e$}\end{cases}\]
 For each $g\in\Gamma$, there is a canonical conditional expectation $\mathbb{E}_g:C_r^*(\Gamma)\to C_r^*(\langle g \rangle)$ defined by \[\mathbb{E}_g\left(\lambda(s)\right)=\begin{cases}
\lambda(s) & \mbox{if $s\in\langle g \rangle$}\\
0 & \mbox{otherwise}\end{cases}\]

\subsection{Approximation property for groups}
In general, an arbitrary element $a\in C_r^*(\Gamma)$ can not be written as an infinite sum $\sum_{g\in\Gamma} a(g)\lambda(g)$. But when the group satisfies the approximation property introduced by Haagerup, such a convergence (in the $\|\cdot\|$-topology) for a weighted sum is known to hold (see, for example, \cite{suzuki2017group,crann2022non}).
For a finitely supported function $\phi$ on $\Gamma$, let us denote  by $m_{\phi}: C_r^*(\Gamma) \to C_r^*(\Gamma)$ the completely bounded map defined by the formula $$m_{\phi}\left(\sum_{s}a_s\lambda_s\right)=\sum_sa_s\phi(s)\lambda_s.$$
A discrete group $\Gamma$ is said to have the approximation property (AP) if there exists a net $\left(\phi_{i}\right)_{i \in I}$ of finitely supported complex-valued functions on $\Gamma$ such that the map $m_{\phi_i} \otimes_{\text{min}} \text{id}_\mathcal{B}: C^*_r(\Gamma) \otimes_{\min} \mathcal{B} \to C^*_r(\Gamma) \otimes_{\min} \mathcal{B}$ converges to the identity map in the pointwise norm topology for any unital $C^*$-algebra $\mathcal{B}$. We refer the reader to \cite[Chapter~12]{BroOza08} for more details on these. We remark that property-AP is closed under extensions (see~\cite[Proposition~12.4.10]{BroOza08}). Hence, it is closed under taking products.
\subsection{Support of a Subalgebra} Now, given an invariant $C^*$-subalgebra $\mathcal{A}\le C_r^*(\Gamma)$, assume that there exists a normal subgroup $N\trianglelefteq\Gamma$ such that $\text{Supp}(a)\subset N$ for all $a\in\mathcal{A}$. Here, $\text{Supp}(a)=\{s\in\Gamma:\tau_0(a\lambda(s)^*)\ne 0\}$. If the group $\Gamma$ has property-AP, then in this case, it is easy to see that $\mathcal{A}\le C_r^*(N)$. Indeed, if $\Gamma$ has property-AP, using \cite[Proposition~3.4]{suzuki2017group}, every element $a\in\mathcal{A}$ can be written as
\[a=\lim_i\sum_{g\in\Gamma}\phi_i(g)\tau_0(a\lambda(g)^*)\lambda(g)\]
Here, $\phi_i$ is a finitely supported function on $\Gamma$. However, even if the group does not have property-AP, then it is still true that $\mathcal{A}\le C_r^*(N)$ as the following proposition shows. We thank Hanfeng Li for showing it to us.
\begin{prop}
\label{prop:supoprtofsubalg}Let $\mathcal{A}\le C_r^*(\Gamma)$ be a $C^*$-subalgebra. Assume that there exists a subgroup $N\le\Gamma$ such that $\text{Supp}(a)\subset N$ for all $a\in\mathcal{A}$. Then, $\mathcal{A}\le C_r^*(N)$.
\end{prop}

\begin{proof}
Let $\mathbb{E}_N:C_r^*(\Gamma)\to C_r^*(N)$ be the canonical conditional expectation. For any $a\in C_r^*(\Gamma)$, we have the following commutative diagram.
\begin{center}
   \begin{tikzcd}[row sep=large,column sep=huge]
\ell^2(N) \arrow{r}{\mathbb{E}_N(a)} \arrow[hookrightarrow]{d}{i_N} & \ell^2(N) \arrow[hookrightarrow]{d}{i_N}
\\
    \ell^2(\Gamma) \arrow[rightarrow]{r}{a} & \ell^2(\Gamma)
\end{tikzcd}
\end{center}
Let $I$ be a collection of right coset representatives for $N$ in $\Gamma$.
We view $C_r^*(N)\hookrightarrow C_r^*(\Gamma)\curvearrowright\ell^2(\Gamma)=\bigoplus_{s\in I}\ell^2(Ns)$, which is
explicitly given by
\[a\xi:=[a(\xi s^{-1})]s,~\forall~\xi\in \ell^2(Ns).\]
Let $P_N$ denote the projection from $\ell^2(\Gamma)$ onto $\ell^2(N)$. Note that $\mathbb{E}_N(a)=P_N\circ a\circ i_N$ for all $a\in C_r^*(\Gamma)$. Now, if $a\in C_r^*(\Gamma)$ is such that $\text{Supp}(a)\subseteq N$, we claim that $\mathbb{E}_N(a)=a$. Observe that for all $a\in C_r^*(\Gamma)$ and $\xi\in\ell^2(\Gamma)$,
\[a(\xi s^{-1})=a(\rho(s)\xi)=\rho(s)(a\xi)=(a\xi)s^{-1}, ~s\in\Gamma.\]
Now, since $\mathbb{E}_N(a)(\xi)=[P_Na(\xi s^{-1})]s$, it is enough to show that $a(\xi s^{-1})\subset\ell^2(N)$ for all $\xi\in \ell^2(Ns)$. Now, for all $n_1, n_2\in N$ and for all $s,t\in\Gamma$ with $ts^{-1}\not\in N$,
\begin{align*}
\langle a(n_1s), n_2t\rangle&=\tau_0\left(an_1st^{-1}n_2^{-1}\right)\\&=\tau_0(an_1n_2^{-1}(n_2st^{-1}n_2^{-1}))\\&=   \tau_0\left(an_1n_2^{-1}\tilde{g}_2\right)~~(\tilde{g}_2:=n_2st^{-1}n_2^{-1}\not\in N)
\end{align*}
Since $\text{Supp}(a)\subset N$ and $n_1n_2^{-1}\tilde{g}_2\not\in N$, we see that
\[\langle a(n_1s), n_2t\rangle=\tau_0\left(an_1n_2^{-1}\tilde{g}_2\right)=0.\]
The claim follows.
\end{proof}
\subsection{General Strategy}\label{subsection: general strategy}
Given an invariant $C^*$-subalgebra $\mathcal{A}\le C_r^*(\Gamma)$, it is not clear a prior that $\mathcal{A}$ is the image of a conditional expectation. As such, the \say{bimodule comparing technique} employed by the second named author in \cite{jiangskalski2021maximal,jiang2021maximal,jiang2020maximal} and exploited fully in \cite{amrutam2023invariant} is not applicable. We work with the canonical conditional expectations $\{\mathbb{E}_t: t\in\Gamma\}$ to circumvent this issue. Given $a\in\mathcal{A}$ and $t\in\Gamma$, we compare $\mathbb{E}_t(a)$ with $\mathbb{E}_s(a)$ for a suitably chosen $s$. The following proposition makes this idea precise and outlines the strategy.
\begin{prop}
\label{prop:genstrategy}
Let $\Gamma$ be a torsion-free discrete group. Let $\mathcal{A}\le C_r^*(\Gamma)$ be a $\Gamma$-invariant C$^*$-subalgebra. Assume that the following two conditions hold:
\begin{enumerate}
    \item $\mathbb{E}_t(\mathcal{A})\subset\mathcal{A}$ for all elements $t\in\Gamma$.
    \item For each non-identity  element $s\in\Gamma$, there exists $e\neq t\in\Gamma$ such that  $s$ is free from $t$ in the sense that $\langle s,t\rangle\cong \langle s\rangle \star\langle t\rangle\cong \mathbb{F}_2$.
\end{enumerate}
Then, $\mathcal{A}=C_r^*(N)$ for some normal subgroup $N\triangleleft\Gamma$.
\end{prop}
Many natural classes of groups satisfy the second assumption of Proposition~\ref{prop:genstrategy}. Recall that a discrete group $\Gamma$ has property $P_{\text{nai}}$ if, for any finite subset $F\subset\Gamma\setminus\{e\}$, there exists $s\in \Gamma$ of infinite order such that $\langle s, t\rangle\cong \langle s \rangle\star\langle t \rangle$ for all $t\in F$ (see \cite[Definition~4]{bekka1994some}).

It was shown in \cite[Theorem~3]{bekka1994some} that every Zariski-dense discrete subgroup $\Gamma$ of a connected simple Lie group of real-rank $1$ and trivial center has property $P_{\text{nai}}$. Moreover, every acylindrically hyperbolic group with no non-trivial finite normal subgroups also has property $P_{\text{nai}}$ (see \cite[Theorem~0.2]{abbott2019property}). The relatively
hyperbolic groups without non-trivial finite normal subgroups also satisfy the
property $P_{\text{nai}}$~\cite{arzhantseva2007relatively}. The groups, acting faithfully and geometrically
on a non-Euclidean, possibly reducible CAT($0$) cube complex, also have property
$P_{\text{nai}}$~\cite{kar2016ping}. More recently, it was shown in \cite{lou2024property} that big mapping class groups have property $P_{\text{nai}}$ in the setting of infinite type surfaces (such groups are not necessarily acylindrically hyperbolic).

That there are groups that satisfy the first assumption of Proposition~\ref{prop:genstrategy} is non-trivial and requires a bit more work. We show that all torsion-free, non-amenable hyperbolic groups satisfy a weak version of this property; this is enough to finish the proof. The following lemma allows us to do so.

\begin{lem}\label{lem: trick to handle primitivity}
Let $\mathbb{F}_2=\langle t,h\rangle$ be the non-abelian free group of rank two with generators $t$ and $h$. Let $n_1,n_2$ be non-zero integers. Assume that  $t^{n_1}ht^{n_2}h^{-1}$ commutes with $t^kht^mh^{-1}$ for some integers $k,m$, then either $k=n_1,m=n_2$ or $k=m=0$.
\end{lem}
\begin{proof}
By assumption, we have
\begin{align}\label{eq: commuting relation identity in free groups}
(t^{n_1}ht^{n_2}h^{-1})(t^kht^mh^{-1})
=(t^kht^mh^{-1})(t^{n_1}ht^{n_2}h^{-1}).
\end{align}
If $k=0$, then $(t^{n_1}ht^{n_2}h^{-1})(t^kht^mh^{-1})=t^{n_1}ht^{n_2+m}h^{-1}$, after possibly putting in reduced form in $\mathbb{F}_2$, is a reduced word with initial letters $t^{n_1}$; Notice that $(t^kht^mh^{-1})(t^{n_1}ht^{n_2}h^{-1})=ht^mh^{-1}t^{n_1}ht^{n_2}h^{-1}$ starts with $t^{n_1}$ only when $m=0$. Therefore, in this case, $k=m=0$ is proved.

We may assume $k\neq 0$ now. Then
the LHS of \eqref{eq: commuting relation identity in free groups}, after a possible reduction, is a reduced word with initial letters $t^{n_1}ht^{n_2}h^{-1}t^k$. Comparing it with the initial letters of the word on the RHS of \eqref{eq: commuting relation identity in free groups}, it is not hard to deduce that $m\neq 0$ and hence also $k=n_1$, $m=n_2$.
\end{proof}
We now prove Proposition~\ref{prop:genstrategy}.
\begin{proof}[Proof of Proposition~\ref{prop:genstrategy}]
Let $\mathcal{A}$ be a non-trivial invariant C$^*$-subalgebra of $C_r^*(\Gamma)$. Let $a\in\mathcal{A}$ be a non-zero element. Without any loss of generality, we may assume $\tau_0(a)=0$ since $\mathbb{C}\subseteq\mathcal{A}$. We aim to show that $\lambda(g)\in\mathcal{A}$ whenever $\tau_0(a\lambda(g)^*)\ne 0$.  Indeed, assume  this is proved, then set $N=\{g\in \Gamma:\lambda(g)\in \mathcal{A}\}$. Clearly, $N$ is a normal subgroup in $\Gamma$ and $C^*_r(N)\subseteq \mathcal{A}$. Moreover, since $\text{Supp}(a)\subset\mathcal{N}$ for all $a\in\mathcal{A}$, using Proposition~\ref{prop:supoprtofsubalg}, we will deduce that $\mathcal{A}\subseteq C^*_r(N)$.

Let $a\in\mathcal{A}$ be a non-zero element. Fix $g\in\Gamma$ such that $\tau_0(a\lambda(g)^*)\ne0$. Let $h\in \Gamma$ be such that $g$ is free from $h$, i.e., $\langle g, h\rangle=\langle g\rangle\star\langle h\rangle\cong \mathbb{F}_2$. We will write $g$ instead of $\lambda(g)$ for ease of notation. Let $0<\epsilon<1$ be given. Using the first bit of the assumption, we see that $ \mathbb{E}_{g}(a)\in\mathcal{A}$. Since $\mathbb{E}_g(a)\in C_r^*(\langle g \rangle)$ and the group ring $\mathbb{C}[\langle g\rangle]$ is norm dense in $C_r^*(\langle g \rangle)$, we can find a finite subset $F\subset \mathbb{Z}\setminus\{1\}$ such that
\[\left\|\mathbb{E}_{g}(a)-\sum_{k\in F}c_{g^k}g^k-c_gg\right\|<\frac{\epsilon}{2(1+\|a\|)}.\]

Therefore, using the fact that $\tau_0\left(ag^{-1}\right)=\tau_0\left(\mathbb{E}_{g}(ag^{-1})\right)=\tau_0\left(\mathbb{E}_{g}(a)g^{-1}\right)$, we see that
\[\]
\begin{equation}
\label{eq:firstoneprimit}
  \left\| \mathbb{E}_{g}(a)-\left(\sum_{k\in F}c_{g^k}g^k+\tau_0(ag^{-1})g\right)\right\|<\frac{\epsilon}{(1+\|a\|)}.
\end{equation}
Note that this, in particular, implies
\[\left\|\sum_{k\in F}c_{g^k}g^k+\tau_0(ag^{-1})g\right\|\le\left\|\sum_{k\in F}c_{g^k}g^k+\tau_0(ag^{-1})g-\mathbb{E}_g(a)\right\|+\left\|\mathbb{E}_g(a)\right\|<\epsilon+\|a\|<1+\|a\|.\]
Since $\mathcal{A}$ is $\Gamma$-invariant, $\lambda(h)a\lambda(h)^*\in\mathcal{A}$. Moreover, note that
\begin{align*}
&\mathbb{E}_{hgh^{-1}}(\lambda(h)a\lambda(h)^*)=\mathbb{E}_{hgh^{-1}}(\lambda(h)(a-\mathbb{E}_g(a))\lambda(h)^*)+\mathbb{E}_{hgh^{-1}}(\lambda(h)\mathbb{E}_g(a)\lambda(h)^*)\\&\approx_{\frac{\epsilon}{1+\|a\|}}\mathbb{E}_{hgh^{-1}}\left(\lambda(h)\left(\sum_{\substack{i=1\\g_i\not\in\langle g\rangle}}^mc_{g_i}\lambda(g_i)\right)\lambda(h)^*\right)+\mathbb{E}_{hgh^{-1}}\left(\sum_{k\in F}c_{g^k}hg^kh^{-1}+\tau_0(ag^{-1})hgh^{-1}\right)\\&=\sum_{k\in F}c_{g^k}hg^kh^{-1}+\tau_0(ag^{-1})hgh^{-1}.
\end{align*}
Note that by the first part of our assumption, $ \mathbb{E}_{hgh^{-1}}(\lambda(h)a\lambda(h)^*)\in\mathcal{A}$ and,
\begin{equation}
\label{eq:withoutAPfreeprimit}
    \left\|\mathbb{E}_{hgh^{-1}}(\lambda(h)a\lambda(h)^*)-\sum_{k\in F}c_{g^k}hg^kh^{-1}-\tau_0(ag^{-1})hgh^{-1}\right\|<\frac{\epsilon}{1+\|a\|}.
\end{equation}
Multiplying elements in equation~\eqref{eq:firstoneprimit} with equation~\eqref{eq:withoutAPfreeprimit}, we get that
\begin{equation}
  y:=\mathbb{E}_{g}(a)  \mathbb{E}_{hgh^{-1}}(\lambda(h)a\lambda(h)^*)\approx_{2\epsilon}\sum_{\substack{k,m\in F\cup\{1\} \\ (k,m)\neq (1,1)}}c_{g^k}c_{g^m}g^khg^mh^{-1}+\tau_0(ag^{-1})^2ghgh^{-1}.
\end{equation}
Note that in the above, we have implicitly abused notation by writing $c_g=\tau_0(ag^{-1})$.

Write $ghgh^{-1}=u$ for some element $u\in\Gamma$. Again using the first part of the assumption on $y$ and noticing that $c_e=0$, we claim that $$\mathbb{E}_{u}(y)\approx_{2\epsilon}\tau_0(ag^{-1})^2ghgh^{-1}.$$

To see this, it suffices to argue that
$\langle u\rangle\cap \{g^khg^mh^{-1}: k,m\in\mathbb{Z}\}=\{e,ghgh^{-1}\}$.
Assume that $u^i=g^khg^mh^{-1}$ for some $i\neq 0$ and $k,m\in\mathbb{Z}$, then $ghgh^{-1}=u$ commutes with $g^khg^mh^{-1}$, hence  $k=m=1$ by Lemma \ref{lem: trick to handle primitivity}.
Since $\tau_0(ag^{-1})\ne0$, and $\epsilon>0$ is arbitrary (independent of $a$, $g$ and $h$), we get that $ghgh^{-1}\in\mathcal{A}$. Now, replacing $h$ by $h^2$, and using the same argument (noting that $\langle g, h^2\rangle=\langle g\rangle\star\langle h^2\rangle\cong \mathbb{F}_2$), we get that $gh^2gh^{-2}\in\mathcal{A}$. Consequently,
$$h(hg^{-1}h^{-1}g)h^{-1}=h^2g^{-1}h^{-1}gh^{-1}=h^2g^{-1}h^{-2}hgh^{-1}=(gh^2gh^{-2})^{-1}ghgh^{-1}\in\mathcal{A}.$$
Therefore, $hg^{-1}h^{-1}g\in\mathcal{A}$. Hence,
$$g^2=(ghgh^{-1})(hg^{-1}h^{-1}g)\in\mathcal{A}.$$
Therefore, $\mathbb{E}_g(a)(hg^{2}h^{-1})\in\mathcal{A}$. Observe that
$$\left\|\mathbb{E}_g(a)(hg^{2}h^{-1})-\sum_{k\in F}c_{g^k}g^khg^{2}h^{-1}-\tau_0(ag^{-1})ghg^{2}h^{-1}\right\|<
\epsilon.$$

Write $ghg^{2}h^{-1}=v$ for some element $v\in\Gamma$.
Using Lemma \ref{lem: trick to handle primitivity}, we see that $$\left\|\tau_0(ag^{-1})ghg^{2}h^{-1}-\mathbb{E}_{v}\left(\mathbb{E}_g(a)(hg^{2}h^{-1})\right)\right\|<\epsilon.$$
Note that $\mathbb{E}_{v}\left(\mathbb{E}_g(a)(hg^2h^{-1})\right)\in\mathcal{A}$.
Since $\epsilon>0$ is arbitrary, and independent of $a$, $g$ and $h$, we obtain that $\tau_0(ag^{-1})ghg^{2}h^{-1}\in\mathcal{A}$. Since $\tau_0(ag^{-1})\ne 0$, we see that $ghg^{2}h^{-1}\in\mathcal{A}$. We recall that $ghgh^{-1}\in\mathcal{A}$. Therefore,
$$ghg^{2}h^{-1}(ghgh^{-1})^{-1}=ghg^{2}h^{-1}hg^{-1}h^{-1}g^{-1}=gh g h^{-1}g^{-1}=(gh) g (gh)^{-1}\in\mathcal{A}.$$
Since $\mathcal{A}$ is invariant, it follows that $g\in\mathcal{A}$. The claim follows.
\end{proof}


\section{An averaging principle and its consequences}
We start by reproving the following singularity phenomenon, previously leveraged to establish results on rigidity (refer to works such as~\cite{KK, hartman2023stationary, kalantar2022invariant, bader2021charmenability, amrutam2024subalgebras} among others, for examples). As stated below, the claim has appeared in \cite[Lemma~4.1]{amrutam2024relative}. We restate it with all the details for the sake of completion.
\begin{lemma}
\thlabel{twodelta}
Let $\Gamma\curvearrowright X$ be a continuous action on a compact Hausdorff space $X$. Let $\tau$ be a state on $C(X)\rtimes_r\Gamma$ such that $\tau|_{C(X)}=a\delta_x+(1-a)\delta_y$ for some $x\ne y\in X$. Then,
$\tau(\lambda(s))=0$ for all $s \in \Gamma$ with $s\{x,y\}\cap\{x,y\}=\emptyset$.
\begin{proof}
Let $s\in\Gamma$ be such that the intersection $s\{x,y\}\cap\{x,y\}$ is empty. Leveraging Urysohn's lemma, choose $f$ with $0\le f\le 1$, which has a value of $1$ on the set $\{x,y\}$ and $0$ on the set $\{sx,sy\}$. An application of the Cauchy-Schwarz inequality reveals that:
\begin{align*}
\left|\tau(f\lambda(s))\right|^2= \left|\tau(\sqrt{f}\sqrt{f}\lambda(s))\right|^2\leq \tau(f)\tau(s^{-1}.f).
\end{align*}
Since $s\{x,y\}\cap\{x,y\}=\emptyset$ and the state $\tau$ restricted to $C(X)$ equals $a\delta_x+(1-a)\delta_y$, it follows that $\tau(s^{-1}.f)=0$. This leads to the conclusion that $\tau(f\lambda(s))=0$. The Cauchy-Schwarz inequality is employed once more on $|\tau((1-f)\lambda(s))|^2$ to obtain
\begin{align*}
|\tau\left((1-f)\lambda(s)\right)|^2&=\left|\tau\left(\sqrt{1-f}\sqrt{1-f}\lambda(s)\right)\right|^2\\&\le\tau\left(1-f\right)\tau\left(\lambda(s^{-1})(1-f)\lambda(s)\right)\\&=\tau(1-f)\tau(s^{-1}.(1-f))
\end{align*}
We now observe that $\tau((1-f)\lambda(s))=0$ (this is due to the fact that $f$ is $1$ on $\{x,y\}$, making $\tau(1-f)=0$). It now follows that
$$\tau(\lambda(s))=\tau(f\lambda(s))+\tau((1-f)\lambda(s))=0,$$ which concludes the proof.
\end{proof}
\end{lemma}
The above fact allows us to do an averaging involving elements from a given set rather than a subgroup. This observation has been used recently in \cite{amrutam2024crossed}. Instead of averaging everything to $\epsilon$, we can only average the elements not in the given set's centralizer. We make it precise below. Recall that an action $\Gamma\curvearrowright X$ is said to have \say{north pole-south pole}-dynamics, if for every infinite order element $g\in\Gamma$, there are precisely two fixed points $x_g^+$ and $x_g^-$ on the $\Gamma$-space $X$ such that $g^nx\xrightarrow[]{n\to\infty}x_g^+$ for all $x\ne x_g^{-}$.

\begin{prop}
\label{prop:plclosure}
Suppose that $\Gamma\curvearrowright X$ has north pole-south pole dynamics. Let $s\in\Gamma$ be an infinite order element. Denote by $x_s^+$ and $x_s^-$ the corresponding fixed points. Moreover, assume that $t\{x_s^+,x_s^-\}\cap \{x_s^+,x_s^-\}=\emptyset$ for all $t\not\in\langle s\rangle$.  Let $\{n_j\}_j\subset\mathbb{N}$ be any subsequence. Then, given $a\in C_r^*(\Gamma)$ and $\delta>0$, we can find $\{s_1,s_2,\ldots,s_m\}\subset\{s^{n_j}: j\in\mathbb{N}\}$ such that
\[\left\|\frac{1}{m}\sum_{j=1}^m\lambda(s_j)a\lambda(s_j)^*-\mathbb{E}_{s}(a)\right\|<\delta.\]
Here, $\mathbb{E}_{s}: C_r^*(\Gamma)\to C_r^*(\langle s \rangle)$ is the canonical conditional expectation.
\end{prop}
Our approach to substantiating the claim draws inspiration from the beautiful paper \cite{haagerup2016new}. We define $\Lambda$ as the set $\{s^{n_j}: j\in\mathbb{N}\}$. We then show that given any bounded functional $\omega$ on $C_r^*(\Gamma)$, we can find $\psi\in \overline{\{s.\omega: s\in\Lambda\}}^{\text{weak}^*}$, satisfying the condition $\psi=\psi\circ\mathbb{E}_{\Lambda}$. The assertion is then established by employing a conventional Hahn-Banach separation argument.
\begin{proof}
Given a bounded linear functional $\varphi$ on $C_r^*(\Gamma)$, let us extend it to a bounded linear functional $\eta$ on $C(X)\rtimes_r\Gamma$. Now $\eta$ can be expressed as $\eta=c_1\omega_1-c_2\omega_2+ic_3\omega_3-ic_4\omega_4$, with each $\omega_i\in S(C(X)\rtimes_r\Gamma)$ and each $c_i\in\mathbb{R}$, for $i=1,2,3,4$. Define $\nu_i$ as the restriction of $\omega_i$ to $C(X)$ for each $i$. Given that $s$ is of infinite order, it uniquely determines fixed points $x_s^+$ and $x_s^-$ in $X$, ensuring that $s^n x\xrightarrow[]{n\to\infty} x_s^+$ for all $x\ne x_s^{-}$. In particular, $s^{n_j} x\xrightarrow[]{j\to\infty} x_s^+$ for all $x\ne x_s^{-}$. Applying the dominated convergence theorem, it can be shown that $s^{n_j}\nu_i\xrightarrow{\text{weak}^*}a_i\delta_{x_s^+}+(1-a_i)\delta_{x_s^-}$. Note that $a_i=\nu_i(X\setminus x_s^-)$ for each $i=1,2,3,4$. Assuming, without loss of generality, up to taking a subsequence four times if necessary, $s^{n_j} \omega_i$ converges to some $\omega_i'\in S(C(X)\rtimes_r\Gamma)$ for each $i=1,2,3,4$. Note that $\omega_i'|_{C(X)}=a_i\delta_{x_s^+}+(1-a_i)\delta_{x_s^-}$. Define $\eta'=c_1\omega_1'-c_2\omega_2'+ic_3\omega_3'-ic_4\omega_4'$. Also note that even after passing to a subsequence four times, we still remain part of the given sequence $\Lambda$.

We let $\psi=\eta'|_{C_r^*(\Gamma)}$, and we claim that $\psi=\psi\circ\mathbb{E}_{ s }$. Since $t\{x_s^+,x_s^-\}\cap \{x_s^+,x_s^-\}=\emptyset$ for all $t\not\in\langle s\rangle$, it follows from \thref{twodelta} that $\omega_i'(\lambda(t))=0$ for each $i=1,2,3,4$. This shows that $\psi=\psi\circ\mathbb{E}_{ s }$. Finally, we can prove the claim using a standard Hahn-Banach separation argument (see, for example, \cite[Theorem~3.4]{bryder2018reduced}).
\end{proof}
While dealing with the product, we need a stronger version of the averaging process because we need to average finitely many elements simultaneously.
\begin{theorem}
\label{thm:finitelymanyaveraging}
Suppose that $\Gamma\curvearrowright X$ has north pole-south pole dynamics. Let $s\in\Gamma$ be an infinite order element. Denote by $x_s^+$ and $x_s^-$ the corresponding fixed points. Moreover, assume that $t\{x_s^+,x_s^-\}\cap \{x_s^+,x_s^-\}=\emptyset$ for all $t\not\in\langle s\rangle$.  Let $\{n_j\}_j\subset\mathbb{N}$ be any subsequence. Then, given finitely many elements $a_1,a_2,\ldots,a_n\in C_r^*(\Gamma)$ and $\epsilon>0$, we can find $\{s_1,s_2,\ldots,s_m\}\subset\{s^{n_j}: j\in\mathbb{N}\}$ such that
\[\left\|\frac{1}{m}\sum_{j=1}^m\lambda(s_j)a_i\lambda(s_j)^*-\mathbb{E}_{s}(a_i)\right\|<\epsilon,~\forall i=1,2,\ldots,n.\]
\end{theorem}
\begin{proof} We first show that given a finite sum $T=\sum_{i=1}^nc_i\lambda(s_i)\in C_r^*(\Gamma)$,
\begin{equation}
\label{eq:absnorm}    \left\|\sum_{i=1}^nc_i\lambda(s_i)\right\|\le  \left\|\sum_{i=1}^n|c_i|\lambda(s_i)\right\|.
\end{equation}
Indeed, for any $\xi=\sum_{s\in\Gamma}\xi(s)\delta_s\in\ell^2(\Gamma)$, letting $\tilde{\xi}=\sum_{s\in\Gamma}|\xi(s)|\delta_s$ we see that
\[\left\|T\xi\right\|^2=\sum_{t\in\Gamma}\left|\sum_{i=1}^nc_i\xi(s_i^{-1}t)\right|^2\le\sum_{t\in\Gamma}\left(\sum_{i=1}^n|c_i|\left|\xi(s_i^{-1}t)\right|\right)^2=\left\|\left(\sum_{i=1}^n|c_i|\lambda(s_i)\right)\tilde{\xi}\right\|^2.\]
Therefore, since $\|\xi\|=\|\tilde{\xi}\|$, we get that
\[\left\|T\xi\right\|^2=\left\|\left(\sum_{i=1}^nc_i\lambda(s_i)\right)\xi\right\|^2\le\left\|\left(\sum_{i=1}^n|c_i|\lambda(s_i)\right)\right\|^2\|\xi\|^2. \]
Equation~\eqref{eq:absnorm} follows by taking the supremum over $\xi\in\ell^2(\Gamma)$ with norm $1$.
Let $a_1,a_2,\ldots,a_n\in C_r^*(\Gamma)$ and $\epsilon>0$. We can find a finite set $F_i\subset\Gamma$ such that
\begin{equation}
\label{eq:approxsum}
a_i\approx_{\epsilon}\sum_{t\in F_i}c_t^i\lambda(t),c_j^i\in\mathbb{C},~i=1,2,\ldots,n.
\end{equation}
Consider $\tilde{a}=\sum_{i=1}^n\sum_{t\in F_i}\left|c_t^i\right|\lambda(t)$. Applying Proposition~\ref{prop:plclosure} to $\Tilde{a}$, we can find a subset $\{s_1,s_2,\ldots,s_m\}\subset\{ s^{n_j}:j\in\mathbb{N} \}$ such that
\begin{equation}
\left\|\frac{1}{m}\sum_{j=1}^m\lambda(s_j)\left(\tilde{a}-\mathbb{E}_{s}(\tilde{a})\right)\lambda(s_j)^*\right\|<\frac{\epsilon}{3}.
\end{equation}
This is equivalent to saying that
\[\left\|\sum_{i=1}^n\frac{1}{m}\sum_{j=1}^m\lambda(s_j)\left(\sum_{t\in F_i}|c_t^i|\lambda(t)-\mathbb{E}_{s}\left(\sum_{t\in F_i}|c_t^i|\lambda(t))\right)\right)\lambda(s_j)^*\right\|<\frac{\epsilon}{3}.\]
Using \cite[Lemma~4.1]{haagerup2016new}, we see that
\[\left\|\frac{1}{m}\sum_{j=1}^m\lambda(s_j)\left(\sum_{t\in F_i}|c_t^i|\lambda(t)-\mathbb{E}_{s}\left(\sum_{t\in F_i}|c_t^i|\lambda(t))\right)\right)\lambda(s_j)^*\right\|<\frac{\epsilon}{3},~i=1,2,\ldots,n.\]
Using Equation~\ref{eq:absnorm}, we get that
\begin{equation}
\label{eq:modtonormal}
    \left\|\frac{1}{m}\sum_{j=1}^m\lambda(s_j)\left(\sum_{t\in F_i}c_t^i\lambda(t)-\mathbb{E}_{s}\left(\sum_{t\in F_i}c_t^i\lambda(t))\right)\right)\lambda(s_j)^*\right\|<\frac{\epsilon}{3},~i=1,2,\ldots,n.
\end{equation}
As usual, the triangle inequality gives us the following. For each $i=1,2,\ldots,n$, we get that
\begin{align*}
&\left\|\frac{1}{m}\sum_{j=1}^m\lambda(s_j)\left(a_i-\mathbb{E}_{s}\left(a_i)\right)\right)\lambda(s_j)^*\right\|\\&\le \left\|\frac{1}{m}\sum_{j=1}^m\lambda(s_j)\left(a_i-\sum_{t\in F_i}c_t^i\lambda(t)\right)\lambda(s_j)^*\right\|\\&+\left\|\frac{1}{m}\sum_{j=1}^m\lambda(s_j)\left(\sum_{t\in F_i}c_t^i\lambda(t)-\mathbb{E}_{s}\left(\sum_{t\in F_i}c_t^i\lambda(t))\right)\right)\lambda(s_j)^*\right\|\\&+\left\|\frac{1}{m}\sum_{j=1}^m\lambda(s_j)\left(\mathbb{E}_{s}\left(\sum_{t\in F_i}c_t^i\lambda(t))-a_i\right)\right)\lambda(s_j)^*\right\|\\&\le\left\|a_i-\sum_{t\in F_i}c_t^i\lambda(t)\right\|+ \left\|\frac{1}{m}\sum_{j=1}^m\lambda(s_j)\left(\sum_{t\in F_i}c_t^i\lambda(t)-\mathbb{E}_{s}\left(\sum_{t\in F_i}c_t^i\lambda(t))\right)\right)\lambda(s_j)^*\right\|\\&+\left\|\mathbb{E}_{s}\left(\sum_{t\in F_i}c_t^i\lambda(t)-a_i\right)\right\|\\&\stackrel{\eqref{eq:approxsum}}{\le}\frac{\epsilon}{3}\stackrel{\eqref{eq:modtonormal}}{+}\frac{\epsilon}{3}\stackrel{\eqref{eq:approxsum}}{+}\frac{\epsilon}{3}\\&=\epsilon.
\end{align*}
\end{proof}
\section{Invariant subalgebras of torsion-free Hyperbolic groups}
Let $\Gamma$ be a torsion-free non-amenable hyperbolic group, and $\partial\Gamma$, the associated Gromov boundary. It is well-known that for every non-identity element $s\in\Gamma$, there are exactly two fixed points $x_s^+, x_s^-\in\partial\Gamma$. Moreover, $s^nx\xrightarrow[]{n\to\infty} x_s^+$ for all $x\ne x_s^-\in \partial\Gamma$. Therefore, using dominated convergence theorem, for every $\nu\in \text{Prob}(\partial\Gamma)$, we see that $s^{n}\nu\xrightarrow{\text{weak}^*}a\delta_{x_s^+}+(1-a)\delta_{x_s^-}$. We note that in this case, $a=\nu(X\setminus x_s^-)$.

\subsection{Selective Averaging} Usually, an \say{averaging principle} kills off all non-identity elements. But this does not serve our purpose. As such, we need to average by powers of a given element and kill off all other elements. The following lemma allows us to do so.
\begin{lemma}\label{lem: setwise stablizer of primitive elements}
    Let $\Gamma$ be a torsion-free non-amenable hyperbolic group, and $\partial\Gamma$, the associated Gromov boundary. Let $s$ be a (non-trivial) primitive element in $\Gamma$ in the sense that if $s=t^n$ for some $t\in\Gamma$ and $n\in \mathbb{Z}$, then $n=\pm 1$. Let $\{x_s^+,x_s^-\}$ be the corresponding fixed points on the boundary. Then, $h\{x_s^+,x_s^-\}\cap \{x_s^+,x_s^-\}=\emptyset$ for all $h\not\in\langle s\rangle$.
\end{lemma}
\begin{proof}
    This follows from some facts used in proving \cite[Lemma 2.2]{BC_gafa}. Indeed, let $\Lambda$ be a maximal amenable subgroup containing $s$. As mentioned in this proof, $\Lambda$ is virtually cyclic. Since $\Gamma$ is assumed to be torsion-free, this implies $\Lambda$ is actually infinite cyclic, i.e. $s\in \Lambda\cong \mathbb{Z}=\langle g\rangle$. So $s=g^n$ for some $n\neq 0$. Since $s$ is primitive, we get that $n=\pm 1$, in other words, $\Lambda=\langle s\rangle$. Then checking the proof of part (i) in \cite[Lemma 2.2]{BC_gafa}, we found that $\langle s\rangle=\Lambda=Stab_{\Gamma}(\{x_s^+, x_s^-\})$ and by part (iii) in \cite[Lemma 2.2]{BC_gafa}, we get that $h\{x_s^+, x_s^-\}\cap \{x_s^+, x_s^-\}=\emptyset$ for all $h\not\in \langle s\rangle$.
\end{proof}

\begin{remark}
\label{rem:maximalamenable}
Note that the above proof shows that for every non-trivial element $h$ in a torsion-free non-amenable hyperbolic group $\Gamma$, we may write $h=s^n$ for some primitive element $s\in \Gamma$ and some integer $n$. Indeed, let $\Lambda$ be a maximal amenable subgroup containing $h$. Then $h\in \Lambda\cong \mathbb{Z}=\langle s\rangle$ for some $s\in\Gamma$. Note that $s$ is primitive since otherwise, we may write $s=g^k$ for some $k\in\mathbb{Z}\setminus \{\pm 1\}$ and thus $\Lambda=\langle s\rangle \subsetneq \langle g\rangle$, which  contradicts to the maximal amenability of $\Lambda$.
\end{remark}
Putting all of these together, we obtain the following easy corollary. This allows us to play ping-pong to find the group elements in algebra. Recall that we denote by $\mathbb{E}_{g}$ the canonical conditional expectation from $C_r^*(\Gamma)$ to $C_r^*(\langle g \rangle)$.

\begin{prop}
\label{prop:imageofanelementunderconditionalexpectation}
Let $\Gamma$ be a torsion-free non-amenable hyperbolic group. Let $\mathcal{A}\le C_r^*(\Gamma)$ be a $\Gamma$-invariant C$^*$-subalgebra. Then, $\E_{g}(\mathcal{A})\subset \mathcal{A}$ for every (non-trivial) primitive element $g\in\Gamma$.
\begin{proof}
Let $\mathcal{A}\le C_r^*(\Gamma)$ be an invariant $C^*$-subalgebra and $a\in\mathcal{A}$. Let $e\neq t\in\Gamma$ be a primitive element. Let $\epsilon>0$. Since $t$ is primitive, we know that $\text{Stab}_{\Gamma}\{x_t^+, x_{t}^{-}\}=\langle t \rangle$ by Lemma \ref{lem: setwise stablizer of primitive elements}.
Using Proposition~\ref{prop:plclosure}, we can find $t_1, t_2, \ldots,t_m\in \Lambda=\{t^{k}: k\ge 0\}$ such that
\begin{equation}
\label{eq:fromaveraging}
\left\|\frac{1}{m}\sum_{j=1}^m\lambda(t_j)a\lambda(t_j)^*-\mathbb{E}_{\langle t\rangle}\left(a\right)\right\|<\epsilon.
\end{equation}
Since $\epsilon>0$ is arbitrary, the claim follows.
\end{proof}
\end{prop}
\subsection{Invariant Subalgebras to Normal Subgroups}
We now establish examples of $\Gamma$ such that every invariant $C^*$-subalgebra $\mathcal{A}\le C_r^*(\Gamma)$ comes from a subgroup. Our strategy is a modification of the argument made in \cite[Proposition~2.2]{amrutam2023invariant}. The primary difference is that the invariant $C^*$-subalgebra is, a priori, not necessarily the image of a conditional expectation. We work around this obstruction by using Proposition~\ref{prop:imageofanelementunderconditionalexpectation}.
\begin{theorem}
\label{thm:invfromsubgroup}
Let $\Gamma$ be a torsion-free non-amenable hyperbolic group. Let $\mathcal{A}\le C_r^*(\Gamma)$ be a $\Gamma$-invariant C$^*$-subalgebra. Then, $\mathcal{A}=C_r^*(N)$ for some normal subgroup $N\triangleleft \Gamma$.
\begin{proof}
Let $\mathcal{A}$ be a non-trivial invariant $C^*$-subalgebra of $C_r^*(\Gamma)$.
Let $a\in\mathcal{A}$ be a non-zero element. We aim to show that $\lambda(g)\in\mathcal{A}$ whenever $\tau_0(a\lambda(g)^*)\ne 0$. Indeed, if this is proved, set $N=\{g\in \Gamma:\lambda(g)\in \mathcal{A}\}$. Clearly, $N$ is a normal subgroup in $\Gamma$ and $C^*_r(N)\subseteq \mathcal{A}$. Moreover, since $\text{Supp}(a)\subset N$ for all $a\in\mathcal{A}$, using Proposition~\ref{prop:supoprtofsubalg}, we deduce that $\mathcal{A}\subseteq C^*_r(N)$.

Let $g\in\Gamma$ be such that $\tau_0(a\lambda(g)^*)\ne 0$. Without loss of generality, we assume $\tau_0(a)=0$. Write $g=t^n$ for some primitive element $t\in \Gamma$ and $0\neq n\in\mathbb{Z}$. Let $h\in \Gamma$ be such that $t$ is free from $h$, i.e., $\langle t, h\rangle=\langle t\rangle\star\langle h\rangle\cong \mathbb{F}_2$. We will write $g$ instead of $\lambda(g)$ for ease of notation.

Fix any $0<\epsilon<1$.
Using Proposition~\ref{prop:imageofanelementunderconditionalexpectation}, we know $\mathbb{E}_t(a) \in \mathcal{A}$. We can find a finite subset $F \subset \mathbb{Z}$ containing $n$ and coefficients $c_{t^k}$ such that:
\begin{equation}
\label{eq:oneprimit_eps}
    \left\| \mathbb{E}_{t}(a) - \sum_{k \in F} c_{t^k} t^k \right\| < \frac{\epsilon}{2(1+\|a\|)}.
\end{equation}
Note that $c_{t^n} = \tau_0(a(t^n)^*) \neq 0$.

Since $\mathcal{A}$ is $\Gamma$-invariant, $b := \lambda(h)a\lambda(h)^*\in\mathcal{A}$. Furthermore, by Proposition~\ref{prop:imageofanelementunderconditionalexpectation}, $\mathbb{E}_{hth^{-1}}(b) \in \mathcal{A}$. Observe that:
\begin{align*}
\mathbb{E}_{hth^{-1}}(b) &= \mathbb{E}_{hth^{-1}}(h \mathbb{E}_t(a) h^{-1}) + \mathbb{E}_{hth^{-1}}(h (a-\mathbb{E}_t(a)) h^{-1}) \\
&= \mathbb{E}_{hth^{-1}}(h \mathbb{E}_t(a) h^{-1}) \\
&\approx_{\frac{\epsilon}{2(1+\|a\|)}} \mathbb{E}_{hth^{-1}}\left( h \left( \sum_{k \in F} c_{t^k} t^k \right) h^{-1} \right) \\
&= \sum_{k \in F} c_{t^k} h t^k h^{-1}.
\end{align*}
Thus, we have:
\begin{equation}
\label{eq:freeprimit_eps}
    \left\| \mathbb{E}_{hth^{-1}}(b) - \sum_{k \in F} c_{t^k} h t^k h^{-1} \right\| < \frac{\epsilon}{2(1+\|a\|)}.
\end{equation}

Multiplying the element in \eqref{eq:oneprimit_eps} by the element in \eqref{eq:freeprimit_eps}, and letting $y := \mathbb{E}_{t}(a) \mathbb{E}_{hth^{-1}}(b) \in \mathcal{A}$, we obtain:
\begin{equation}
    \left\| y - \sum_{k, m \in F} c_{t^k} c_{t^m} t^k h t^m h^{-1} \right\| < \epsilon.
\end{equation}

Write $t^n h t^n h^{-1} = u^{\ell}$ for some primitive element $u \in \Gamma$ and some integer $\ell \neq 0$. Using Proposition~\ref{prop:imageofanelementunderconditionalexpectation}, $\mathbb{E}_u(y) \in \mathcal{A}$. We now apply $\mathbb{E}_u$ to the approximation. Using Lemma~\ref{lem: trick to handle primitivity}, for $k, m \in F$, the element $t^k h t^m h^{-1}$ commutes with $u^\ell = t^n h t^n h^{-1}$ if and only if $k=m=n$ (since $F$ is a subset of integers and we may assume $0 \notin F$ as $\tau_0(a)=0$).
Therefore:
\begin{equation}
    \left\| \mathbb{E}_u(y) - c_{t^n}^2 t^n h t^n h^{-1} \right\| < \epsilon.
\end{equation}
Since $\epsilon > 0$ is arbitrary and independent of the group elements, and $c_{t^n} \neq 0$, we conclude that $t^n h t^n h^{-1} \in \mathcal{A}$.

Now, replacing $h$ by $h^2$ (noting $\langle t, h^2 \rangle \cong \mathbb{F}_2$), the same argument yields $t^n h^2 t^n h^{-2} \in \mathcal{A}$. Consequently:
\[ h(h t^{-n} h^{-1} t^n)h^{-1} = (t^n h^2 t^n h^{-2})^{-1} (t^n h t^n h^{-1}) \in \mathcal{A}. \]
Thus $h t^{-n} h^{-1} t^n \in \mathcal{A}$, which implies $t^{2n} = (t^n h t^n h^{-1})(h t^{-n} h^{-1} t^n) \in \mathcal{A}$.

Consider $z := \mathbb{E}_t(a) (h t^{2n} h^{-1}) \in \mathcal{A}$. Using the approximation for $\mathbb{E}_t(a)$:
\[ \left\| z - \sum_{k \in F} c_{t^k} t^k h t^{2n} h^{-1} \right\| < \epsilon. \]
Let $v$ be a primitive element such that $v^{\ell'} = t^n h t^{2n} h^{-1}$ for some non-zero integer $\ell'$. Applying $\mathbb{E}_v$ and using Lemma~\ref{lem: trick to handle primitivity} again (noting that $t^k h t^{2n} h^{-1}$ commutes with $t^n h t^{2n} h^{-1}$ only when $k=n$), we get:
\[ \left\| \mathbb{E}_v(z) - c_{t^n} t^n h t^{2n} h^{-1} \right\| < \epsilon. \]
Thus $t^n h t^{2n} h^{-1} \in \mathcal{A}$ since $\mathbb{E}_v(z)\in \mathcal{A}$ and $0<\epsilon<1$ is arbitrary. Finally:
\[ t^n h t^{2n} h^{-1} (t^n h t^n h^{-1})^{-1} = t^n h t^n h^{-1} \in \mathcal{A} \implies (t^n h) t^n (t^n h)^{-1} \in \mathcal{A}. \]
By invariance, $g = t^n \in \mathcal{A}$. The proof is complete.
\end{proof}
\end{theorem}
\begin{remark} This shows that the reduced group $C^*$-algebras of the torsion-free non-amenable hyperbolic groups have no non-trivial invariant Cartan subalgebras. But this follows in general for all groups with trivial amenable radical from \cite{amrutam2025amenable}. Indeed, say $\mathcal{A}\le C^*_r(\Gamma)$ is a $\Gamma$-invariant nuclear $C^*$-subalgebra, where $\Gamma$ has trivial amenable radical. Then $\overline{\mathcal{A}}^{\text{WOT}}:=\mathcal{M}$ is a $\Gamma$-invariant amenable von Neumann algebra of $L(\Gamma)$. However, using \cite[Theorem~A]{amrutam2025amenable}, we see that $\mathcal{M}=\mathbb{C}$. Hence, $\mathcal{A}=\mathbb{C}$.
\end{remark}
\subsection{On torsion-free acylindrically hyperbolic groups}\label{sub:torsionacylin}An action of a group $\Gamma$ on a metrizable space $(X, d)$ is called \say{acylindrical} if, for any $\epsilon > 0$, there exist positive constants $\delta$ and $N$ such that for any points $x, y \in X$ with $d(x, y) \geq \delta$, the number of elements $g \in \Gamma$ satisfying both $d(x, gx) \leq \epsilon$ and $d(y, gy) \leq \epsilon$ is at most $N$. A group $\Gamma$ is termed as \say{acylindrically hyperbolic} if it admits a non-elementary acylindrical action on a hyperbolic space.

All non-elementary hyperbolic groups are examples of acylindrically hyperbolic groups. Additional instances of acylindrically hyperbolic groups include non-(virtually) cyclic groups that are hyperbolic relative to proper subgroups, $\text{Out}(F_n)$ for $n > 1$, many mapping class groups, as well as non-(virtually cyclic) groups that act properly on proper CAT(0) spaces and include rank one elements. For further examples and details, see \cite[Section~8]{osin2016acylindrically} and its references.

For a group $\Gamma$ acting on a hyperbolic space $S$, an element $g \in \Gamma$ is classified as loxodromic if it has exactly two fixed points, denoted $x_g^+$ and $x_g^-$, on the Gromov boundary $\partial S$, and if for every point $x \in S$, $g^n x$ converges to $x_g^+$, except when $x$ is equal to $x_g^-$. Moreover, there exists a unique maximal virtually cyclic subgroup $E(g) \le \Gamma$ that contains $g$. More specifically, $E(g) = \text{Stab}_{\Gamma}(\{x_g^+, x_g^-\})$, which represents the set wise stabilizer of the points $\{x_g^+, x_g^-\}$ (refer to \cite[Lemma~6.5]{dahmani2017hyperbolically}).

Given a primitive loxodromic element $g$ in a torsion-free acylindrically hyperbolic group with trivial amenable radical, it is well-known that $E(g)=\langle g \rangle$. For every loxodromic element $g\in\Gamma$, there exists a primitive element $g_0\in\Gamma$ and $m\in\mathbb{Z}\setminus\{0\}$ such that $g=g_0^m$ (the same argument in Remark~\ref{rem:maximalamenable} goes through). Arguing similarly as in Theorem~\ref{thm:invfromsubgroup}, we can conclude that $\lambda(g)\in\mathcal{A}$ if there exists $a\in\mathcal{A}$ with $\tau_0(a\lambda(g)^{*})\ne 0$. Therefore, we obtain that every torsion-free acylindrically hyperbolic group satisfies the $C^*$-ISR property.

\section{Invariant subalgebras of a finite product of torsion-free hyperbolic groups}
In this section, we show that a finite product of torsion-free non-amenable hyperbolic groups has the $C^*$-ISR-property. We first begin with a straightforward modification of the averaging principle.
\subsection{Invariance under coordinate projections}
Let us fix some notation first. Given $g\in\Gamma=\Gamma_1\times\Gamma_2\times\cdots\times\Gamma_n$, let us write $g=(s_1,s_2,\ldots,s_n)$. We shall identify $C_r^*(\Gamma)$ with $C_r^*(\Gamma_1)\otimes_{\text{min}}C_r^*(\Gamma_2)\otimes_{\text{min}}\cdots\otimes_{\text{min}}C_r^*(\Gamma_n)$. By an abuse of notation, $\mathbb{E}_{s_1}\otimes\mathbb{E}_{s_2}\otimes\cdots\otimes\mathbb{E}_{s_n}$ will denote the canonical conditional expectation
\begin{align*}\mathbb{E}_{s_1}\otimes\mathbb{E}_{s_2}\otimes\cdots\otimes\mathbb{E}_{s_n}&:C_r^*(\Gamma_1)\otimes_{\text{min}}C_r^*(\Gamma_2)\otimes_{\text{min}}\cdots\otimes_{\text{min}}C_r^*(\Gamma_n)\\&\xrightarrow{}C_r^*(\langle s_1\rangle)\otimes_{\text{min}}C_r^*(\langle s_2\rangle)\otimes_{\text{min}}\cdots\otimes_{\text{min}}C_r^*(\langle s_n\rangle).\end{align*}
Moreover, given a unital completely positive map $\psi_1:C_r^*(\Gamma_1)\to C_r^*(\Gamma_1)$, we extend it to a unital completely positive map
$$\psi_1\otimes\text{id}:C_r^*(\Gamma_1)\otimes_{\text{min}}C_r^*(\Gamma_2)\to C_r^*(\Gamma_1)\otimes_{\text{min}}C_r^*(\Gamma_2).$$
We do it analogously for the other coordinate. We show below that $\mathcal{A}$ is invariant under the coordinate projections $\mathbb{E}_{s_1}\otimes\mathbb{E}_{s_2}\otimes\cdots\otimes\mathbb{E}_{s_n}$ whenever each $s_i$ is a primitive element.
\begin{prop}
\label{prop:modificationaveraging}
Let $\Gamma=\Gamma_1\times\Gamma_2\times\cdots\times\Gamma_n$ be a finite product of groups where each $\Gamma_i$ is a torsion-free non-amenable hyperbolic group. Let $\mathcal{A}\le C_r^*(\Gamma)$ be a $\Gamma$-invariant C$^*$-subalgebra. Let $g=(s_1,s_2,\ldots,s_n)\in\Gamma$ be such that each $s_i$ is a primitive element. Then,  $$\mathbb{E}_{s_1}\otimes\mathbb{E}_{s_2}\otimes\cdots\otimes\mathbb{E}_{s_n}(\mathcal{A})\subset\mathcal{A}.$$
\begin{proof}
We prove it for $n=2$, then an induction argument gives the desired result for any finite product.
Let $\mathcal{A}\le C_r^*(\Gamma)$ be a $\Gamma$-invariant $C^*$-subalgebra, and $g=(s_1,s_2)$ be such that each $s_i$ a primitive element.
Let $a\in\mathcal{A}$ and $\epsilon>0$ be given. We identity $C_r^*(\Gamma)$ with $C_r^*(\Gamma_1)\otimes_{\text{min}}C_r^*(\Gamma_2)$. We can find $a_1,a_2,\ldots,a_k\in C_r^*(\Gamma_1)$ and $b_1,b_2,\ldots,b_k\in C_r^*(\Gamma_2)$ such that
\begin{equation}
\label{eq:firstapproximation}
 \left\|a-\sum_{i=1}^ka_i\otimes b_i\right\|<\frac{\epsilon}{4}
\end{equation}
Let $\mathbb{M}=\max_{1\le i\le k}\{\|a_i\|,\|b_i\|\}$. Using Theorem~\ref{thm:finitelymanyaveraging}, we can find $t_1, t_2, \ldots,t_m\in\langle s_1\rangle$ such that
$$\left\|\frac{1}{m}\sum_{j=1}^m\lambda(t_j)a_i\lambda(t_j)^*-\mathbb{E}_{s_1}\left(a_i\right)\right\|<\frac{\epsilon}{4k\mathbb{M}},~i=1,2,\ldots,k.$$
Letting $\psi_1(\cdot)=\frac{1}{m}\sum_{j=1}^m\lambda(t_j)(\cdot)\lambda(t_j)^*$, the above inequality can be rewritten as
\begin{equation}
\label{eq:firstcoordinate}
\left\|\psi_1(a_i)-\mathbb{E}_{s_1}\left(a_i\right)\right\|<\frac{\epsilon}{4k\mathbb{M}},~i=1,2,\ldots,k.    \end{equation}
Another application of Theorem~\ref{thm:finitelymanyaveraging} and we can find $u_1, u_2, \ldots,u_n\in\langle s_2\rangle$ such that
$$\left\|\frac{1}{n}\sum_{l=1}^n\lambda(u_l)b_i\lambda(u_l)^*-\mathbb{E}_{s_2}\left(b_i\right)\right\|<\frac{\epsilon}{4k\mathbb{M}},~i=1,2,\ldots,k.$$
Letting $\psi_2(\cdot)=\frac{1}{n}\sum_{l=1}^n\lambda(u_l)(\cdot)\lambda(u_l)^*$, we see that
\begin{equation}
\label{eq:secondcoordinate}
\left\|\psi_2(b_i)-\mathbb{E}_{s_2}\left(b_i\right)\right\|<\frac{\epsilon}{4k\mathbb{M}},~i=1,2,\ldots,k.   \end{equation}
From equation~\eqref{eq:firstapproximation}, we get that
\begin{equation}
\label{eq:secondapprox}
\left\|\psi_1\otimes\psi_2(a)-\sum_{i=1}^k\psi_1(a_i)\otimes \psi_2(b_i)\right\|<\frac{\epsilon}{4}.
\end{equation}
Also, applying $\mathbb{E}_{s_1}\otimes\mathbb{E}_{s_2}$ to equation~\eqref{eq:firstapproximation}, we get that
\begin{equation}
\label{eq:thirdapprox}
\left\|\mathbb{E}_{s_1}\otimes\mathbb{E}_{s_2}(a)-\sum_{i=1}^k\mathbb{E}_{s_1}(a_i)\otimes \mathbb{E}_{s_2}(b_i)\right\|<\frac{\epsilon}{4}.
\end{equation}
The triangle inequality gives us the following estimate.
\begin{align*}
&\left\|\psi_1\otimes\psi_2(a)-\mathbb{E}_{s_1}\otimes\mathbb{E}_{s_2}(a)\right\|\\&\le  \left\|\psi_1\otimes\psi_2(a)-\sum_{i=1}^k\psi_1(a_i)\otimes \psi_2(b_i)\right\|+\left\|\sum_{i=1}^k\psi_1(a_i)\otimes \psi_2(b_i)-\sum_{i=1}^k\mathbb{E}_{s_1}(a_i)\otimes \psi_2(b_i)\right\|\\&+\left\|\sum_{i=1}^k\mathbb{E}_{s_1}(a_i)\otimes \psi_2(b_i)-\sum_{i=1}^k\mathbb{E}_{s_1}(a_i)\otimes \mathbb{E}_{s_2}(b_i)\right\|+\left\|\sum_{i=1}^k\mathbb{E}_{s_1}(a_i)\otimes \mathbb{E}_{s_2}(b_i)-\mathbb{E}_{s_1}\otimes\mathbb{E}_{s_2}(a)\right\|\\&\le  \left\|\psi_1\otimes\psi_2(a)-\sum_{i=1}^k\psi_1(a_i)\otimes \psi_2(b_i)\right\|+\left\|\sum_{i=1}^k\left(\psi_1(a_i)-\mathbb{E}_{s_1}(a_i)\right)\otimes \psi_2(b_i)\right\|\\&+\left\|\sum_{i=1}^k\mathbb{E}_{s_1}(a_i)\otimes \left(\psi_2(b_i)-\mathbb{E}_{s_2}(b_i)\right)\right\|+\left\|\sum_{i=1}^k\mathbb{E}_{s_1}(a_i)\otimes \mathbb{E}_{s_2}(b_i)-\mathbb{E}_{s_1}\otimes\mathbb{E}_{s_2}(a)\right\|\\&\le \left\|\psi_1\otimes\psi_2(a)-\sum_{i=1}^k\psi_1(a_i)\otimes \psi_2(b_i)\right\|+\sum_{i=1}^k\left\|\left(\psi_1(a_i)-\mathbb{E}_{s_1}(a_i)\right)\otimes \psi_2(b_i)\right\|\\&+\sum_{i=1}^k\left\|\mathbb{E}_{s_1}(a_i)\otimes \left(\psi_2(b_i)-\mathbb{E}_{s_2}(b_i)\right)\right\|+\left\|\sum_{i=1}^k\mathbb{E}_{s_1}(a_i)\otimes \mathbb{E}_{s_2}(b_i)-\mathbb{E}_{s_1}\otimes\mathbb{E}_{s_2}(a)\right\|\\&\le \left\|\psi_1\otimes\psi_2(a)-\sum_{i=1}^k\psi_1(a_i)\otimes \psi_2(b_i)\right\|+\sum_{i=1}^k\left\|\left(\psi_1(a_i)-\mathbb{E}_{s_1}(a_i)\right)\right\|\left\| \psi_2(b_i)\right\|\\&+\sum_{i=1}^k\left\|\mathbb{E}_{s_1}(a_i)\right\|\left\|\left(\psi_2(b_i)-\mathbb{E}_{s_2}(b_i)\right)\right\|+\left\|\sum_{i=1}^k\mathbb{E}_{s_1}(a_i)\otimes \mathbb{E}_{s_2}(b_i)-\mathbb{E}_{s_1}\otimes\mathbb{E}_{s_2}(a)\right\|\\&\stackrel{\eqref{eq:secondapprox}}{\le}\frac{\epsilon}{4}\stackrel{\eqref{eq:firstcoordinate}}{+}\sum_{i=1}^k\frac{\epsilon}{4\mathbb{M}k}\left\| \psi_2(b_i)\right\|\stackrel{\eqref{eq:secondcoordinate}}{+}\sum_{i=1}^k\left\|\mathbb{E}_{s_1}(a_i)\right\|\frac{\epsilon}{4\mathbb{M}k}\stackrel{\eqref{eq:thirdapprox}}{+}\frac{\epsilon}{4}\\&\le \frac{\epsilon}{4}+\sum_{i=1}^k\frac{\epsilon}{4k}+\sum_{i=1}^k\frac{\epsilon}{4k}+\frac{\epsilon}{4}\\&=\epsilon.
\end{align*}
Since $\epsilon>0$ is arbitrary and $\psi_1\otimes\psi_2(a)\in\mathcal{A}$, the claim follows.
\end{proof}
\end{prop}
\subsection{Need for reduction}
Invariance under the coordinate projection is not enough to guarantee the C$^*$-ISR-property. Since our proof relies on induction, we need to be able to reduce the $n$-th step to the $(n-1)$th step. To do this, we take the help of the slice map with respect to the canonical trace $\tau_0$. Given $\Gamma=\Gamma_1\times\Gamma_2\times\cdots\times\Gamma_n$, let $\tau_0^i$ denote the canonical trace on $C_r^*(\Gamma_i)$ for each $i=1,2,\ldots,n$. By an abuse of notation, $\mathbb{E}_{\tau_0^i}:C_r^*(\Gamma_1)\otimes_{\text{min}} C_r^*(\Gamma_2)\otimes_{\text{min}}\cdots\otimes_{\text{min}} C_r^*(\Gamma_n)\to C_r^*(\Gamma_1)\otimes_{\text{min}} C_r^*(\Gamma_2)\otimes_{\text{min}}\cdots\otimes_{\text{min}} C_r^*(\Gamma_n)$ is defined by
\[\mathbb{E}_{\tau_0^i}(x_1\otimes x_2\otimes\cdots\otimes x_n)=x_1\otimes\cdots \otimes x_{i-1}\otimes\tau_0(x_i)\otimes x_{i+1}\otimes\cdots \otimes x_n.\]
\begin{theorem}
\label{thm:invunderslice}
Let  $n\geq 2$ and $\Gamma=\Gamma_1\times\Gamma_2\times\cdots\times\Gamma_n$, where each $\Gamma_i$ is a $C^*$-simple group. Let $\mathcal{A}\le C_r^*(\Gamma)$ be a $\Gamma$-invariant C$^*$-subalgebra. Then, $\mathbb{E}_{\tau_0^i}(\mathcal{A})\subset\mathcal{A}$ for each $i=1,2,\ldots,n$.
\end{theorem}
\begin{proof}
We do it for $n=2$ for ease of notation. Without any loss of generality, we can look at $\mathbb{E}_{\tau_0^2}$. Let $a\in\mathcal{A}$ and $\epsilon>0$ be given. We can find finite sets $F_1\times F_2\subset\Gamma_1\times\Gamma_2$ and complex numbers $c_{s,t}\in\mathbb{C}$ such that for all $t\in F_2\setminus\{e\}$,
\begin{equation}
\label{eq:nonidentityapprox}
    a\approx_{\frac{\epsilon}{3}}\sum_{(s,t)\in F_1\times(F_2\setminus\{e\})}c_{s,t}\lambda(s)\otimes\lambda(t)+\sum_{s\in F_1}c_{s,e}\lambda(s)\otimes\lambda(e).
\end{equation}
This, in particular, implies that
\begin{equation}
\label{eq:slicenonidentityapprox}
    \mathbb{E}_{\tau_0^2}(a)\approx_{\frac{\epsilon}{3}}\sum_{s\in F_1}c_{s,e}\lambda(s)\otimes\lambda(e).
\end{equation}
Let $M=\max_{(s,t)\in F_1\times F_2}\{|c_{s,t}|\}$. Since $\Gamma_2$ is $C^*$-simple, using \cite[Theorem 5.1]{haagerup2016new}, we can find $\{t_1,t_2,\ldots,t_m\}\subset\Gamma_2$ such that
\begin{equation}
\label{eq:haagerup}
   \left\|\frac{1}{m}\sum_{j=1}^m\lambda(t_jtt_j^{-1})\right\|<\frac{\epsilon}{3M|F_1||F_2|},~\forall t\in F_2\setminus\{e\}.
\end{equation}
Letting $\psi_2:C_r^*(\Gamma_2)\to C_r^*(\Gamma_2)$ be defined by $\psi_2(\cdot)=\frac{1}{m}\sum_{j=1}^m\lambda(t_j)(\cdot)\lambda(t_j)^*$, we see that
\begin{equation}
\label{eq:ucpidentityapprox}
    (id\otimes\psi_2)(a)\approx_{\frac{\epsilon}{3}}\sum_{(s,t)\in F_1\times(F_2\setminus\{e\})}c_{s,t}\lambda(s)\otimes\left(\frac{1}{m}\sum_{j=1}^m\lambda(t_jtt_j^{-1})\right)+\sum_{s\in F_1}c_{s,e}\lambda(s)\otimes\lambda(e).
\end{equation}
Now,
\begin{align*}
&\left\|\mathbb{E}_{\tau_0^2}(a)-(id\otimes\psi_2)(a)\right\|\\&\le\left\|\mathbb{E}_{\tau_0^2}(a)-\sum_{s\in F_1}c_{s,e}\lambda(s)\otimes\lambda(e)\right\|\\&+\left\|\sum_{s\in F_1}c_{s,e}\lambda(s)\otimes\lambda(e)+\sum_{(s,t)\in F_1\times(F_2\setminus\{e\})}c_{s,t}\lambda(s)\otimes\left(\frac{1}{m}\sum_{j=1}^m\lambda(t_jtt_j^{-1})\right)-(id\otimes\psi_2)(a)\right\|\\&+\left\|\sum_{(s,t)\in F_1\times(F_2\setminus\{e\})}c_{s,t}\lambda(s)\otimes\left(\frac{1}{m}\sum_{j=1}^m\lambda(t_jtt_j^{-1})\right)\right\|\\&\stackrel{\eqref{eq:slicenonidentityapprox}}{\le}\frac{\epsilon}{3}\stackrel{\eqref{eq:ucpidentityapprox}}{+}\frac{\epsilon}{3}+ \sum_{(s,t)\in F_1\times(F_2\setminus\{e\})}\left\|c_{s,t}\lambda(s)\otimes\left(\frac{1}{m}\sum_{j=1}^m\lambda(t_jtt_j^{-1})\right)  \right\|\\&\le \frac{2\epsilon}{3}+\sum_{(s,t)\in F_1\times(F_2\setminus\{e\})}|c_{s,t}|\left\|\left(\frac{1}{m}\sum_{j=1}^m\lambda(t_jtt_j^{-1})\right)  \right\|\\&\le \frac{2\epsilon}{3}\stackrel{\eqref{eq:haagerup}}{+}\sum_{(s,t)\in F_1\times(F_2\setminus\{e\})}|c_{s,t}|\frac{\epsilon}{3M|F_1||F_2|}\\&<\epsilon.
\end{align*}
Since $\epsilon>0$ is arbitrary and $(id\otimes\psi_2)(a)\in\mathcal{A}$, the claim follows.
\end{proof}

\begin{cor}\label{cor: mixedsliceandcemaps}
Let  $n\geq 2$ and $\Gamma=\Gamma_1\times\Gamma_2\times\cdots\times\Gamma_n$ be a finite product of groups where each $\Gamma_i$ is a torsion-free non-amenable hyperbolic group. Let $\mathcal{A}\le C_r^*(\Gamma)$ be a $\Gamma$-invariant $C^*$-subalgebra. Let $g=(s_1^{\epsilon_1},s_2^{\epsilon_2},\ldots,s_n^{\epsilon_n})\in\Gamma$ be such that each $s_i$ is a primitive element and $\epsilon_i\in\{0,1\}$ for each $1\leq i\leq n$. Then,  $$\mathbb{E}_{s_1^{\epsilon_1}}\otimes\mathbb{E}_{s_2^{\epsilon_2}}\otimes\ldots\otimes\mathbb{E}_{s_n^{\epsilon_n}}(\mathcal{A})\subset\mathcal{A}.$$
\end{cor}
\begin{proof}

Without loss of generality, we may assume that $\epsilon_i=1$ iff $1\leq i\leq k$ for some $0\leq k\leq n$.

If $k=0$, then $\epsilon_i=0$, i.e. $s_i^{\epsilon_i}=e$ for all $i\geq 1$. Observe that $\mathbb{E}_{e}\otimes \cdots\otimes\mathbb{E}_{e}=\mathbb{E}_{\tau_0^1}\circ \cdots\circ \mathbb{E}_{\tau_0^n}$.
The proof follows by Theorem \ref{thm:invunderslice}. Similarly, the case $k=n$ is treated by Proposition \ref{prop:modificationaveraging}.
From now on, we may assume that $0<k<n$.

Let us do induction on $n$.

For $n=2$. Then $k=1$.
Take any $a\in \mathcal{A}$.
Note that $\mathbb{E}_{s_1}\otimes \mathbb{E}_e=(\mathbb{E}_{s_1}\otimes \text{Id})\circ \mathbb{E}_{\tau_0^2}$. Hence
we may assume $a=\mathbb{E}_{\tau_0^2}(a)\in \mathcal{A}$ by Theorem \ref{thm:invunderslice}. In other words, $a\in \mathcal{A}\cap C^*_r(\Gamma_1)$, which is $\Gamma_1$-invariant. Here we have identified $\Gamma_1$ as the subgroup $\Gamma_1\times \langle e\rangle$ in $\Gamma$.
Therefore, by Proposition \ref{prop:imageofanelementunderconditionalexpectation}, we deduce that $\mathbb{E}_{s_1}\otimes \mathbb{E}_e(a)=\mathbb{E}_{s_1}\otimes \text{Id}(a)\in \mathcal{A}$.
This finishes the proof of case $n=2$.

Assume the proof of case $n-1$ is done,  then notice that for the case $n$,
\begin{align*}
\mathbb{E}_{s_1}\otimes \cdots\otimes \mathbb{E}_{s_k}\otimes \underbrace{\mathbb{E}_e\otimes \cdots\otimes \mathbb{E}_{e}}_{n-k}=(\mathbb{E}_{s_1}\otimes \cdots\otimes \mathbb{E}_{s_k}\otimes \underbrace{\mathbb{E}_e\otimes \cdots\otimes \mathbb{E}_{e}}_{n-k-1}\otimes \text{Id})\circ \mathbb{E}_{\tau_0^n}.
\end{align*}
Given any $a\in \mathcal{A}$. We may assume that $a=\mathbb{E}_{\tau_0^n}(a)\in \mathcal{A}\cap C^*_r(\Gamma_1\times \cdots\times \Gamma_{n-1})$, which is $\Gamma_1\times \cdots\times \Gamma_{n-1}$-invariant. By the induction step, we deduce that $\mathbb{E}_{s_1}\otimes \cdots\otimes \mathbb{E}_{s_k}\otimes \mathbb{E}_e\otimes \cdots\otimes\mathbb{E}_e(a)\in \mathcal{A}$.
\end{proof}
\subsection{\texorpdfstring{$C^*$}{}-ISR for products}
We now prove the main result of this section. The proof is almost the same as in Theorem~\ref{thm:invfromsubgroup}. Nonetheless, we spell out the details.
\begin{theorem}
\label{thm:invfromprod}
Let $n\geq 2$ and $\Gamma=\Gamma_1\times\Gamma_2\times\cdots\times\Gamma_n$ be a finite product of groups where each $\Gamma_i$ is a torsion-free non-amenable hyperbolic group. Let $\mathcal{A}\le C_r^*(\Gamma)$ be a $\Gamma$-invariant $C^*$-subalgebra. Then, $\mathcal{A}=C_r^*(N)$ for some normal subgroup $N\triangleleft\Gamma$.
\end{theorem}
\begin{proof}
We prove the case $n=2$. The general case follows by induction. Let $\mathcal{A} \le C_r^*(\Gamma)$ be a non-trivial invariant subalgebra. Let $a \in \mathcal{A}$ be non-zero. We wish to show that if the coefficient of $\lambda(g) \otimes \lambda(h)$ in the expansion of $a$ is non-zero, then $\lambda(g) \otimes \lambda(h) \in \mathcal{A}$.
Let $c_{g,h} = (\tau_0 \otimes \tau_0)(a (\lambda(g) \otimes \lambda(h))^*)$. Assume $c_{g,h} \neq 0$.

\textit{Case 1: $h=e$ (or symmetrically $g=e$).}
By Corollary~\ref{cor: mixedsliceandcemaps}, we may assume $a \in \mathcal{A} \cap C_r^*(\Gamma_1)$. This reduces to Theorem~\ref{thm:invfromsubgroup}, implying $\lambda(g) \otimes 1 \in \mathcal{A}$.

\textit{Case 2: $g \neq e$ and $h \neq e$.}
Write $g=t_1^n$ and $h=t_2^m$ for primitive elements $t_1 \in \Gamma_1, t_2 \in \Gamma_2$ and integers $n,m \neq 0$. Let $h_1 \in \Gamma_1, h_2 \in \Gamma_2$ be elements such that $\langle t_i, h_i \rangle \cong \mathbb{F}_2$.
Let $0<\epsilon < 1$. By Proposition~\ref{prop:modificationaveraging}, $x := (\mathbb{E}_{t_1} \otimes \mathbb{E}_{t_2})(a) \in \mathcal{A}$.
We can approximate $x$ by a finite sum supported on powers of $t_1$ and $t_2$. There exist finite sets $F_1, F_2 \subset \mathbb{Z}$ (containing $n$ and $m$ respectively) such that:
\begin{equation} \label{eq:prod_approx_1}
    \left\| x - \sum_{k \in F_1, l \in F_2} c_{t_1^k, t_2^l} t_1^k \otimes t_2^l \right\| < \epsilon.
\end{equation}
Let $u = (\lambda(h_1) \otimes \lambda(h_2))$. By invariance, $b := u a u^* \in \mathcal{A}$.
Applying expectations, let $x' := (\mathbb{E}_{h_1 t_1 h_1^{-1}} \otimes \mathbb{E}_{h_2 t_2 h_2^{-1}})(b) \in \mathcal{A}$.
Following the logic of the single variable case, the cross terms vanish under the expectation, yielding:
\begin{equation} \label{eq:prod_approx_2}
    \left\| x' - \sum_{k \in F_1, l \in F_2} c_{t_1^k, t_2^l} (h_1 t_1^k h_1^{-1}) \otimes (h_2 t_2^l h_2^{-1}) \right\| < \epsilon.
\end{equation}
Consider the product $y := x x' \in \mathcal{A}$. Multiplying the approximations:
\[ y \approx_{(2\left\|a\right\|+\epsilon)\epsilon} \sum_{k,p \in F_1} \sum_{l,q \in F_2} c_{t_1^k, t_2^l} c_{t_1^p, t_2^q} (t_1^k h_1 t_1^p h_1^{-1}) \otimes (t_2^l h_2 t_2^q h_2^{-1}). \]
Let $w_1 = t_1^n h_1 t_1^n h_1^{-1}$ and $w_2 = t_2^m h_2 t_2^m h_2^{-1}$. Let $v_1, v_2$ be primitive elements in $\Gamma_1, \Gamma_2$ generating the maximal virtually cyclic subgroups containing $w_1, w_2$.
We apply $(\mathbb{E}_{v_1} \otimes \mathbb{E}_{v_2})$ to $y$.
Using Lemma~\ref{lem: trick to handle primitivity} component-wise, we deduce the following.
\begin{itemize}
    \item In the first coordinate, $t_1^k h_1 t_1^p h_1^{-1}$ commutes with $w_1$ iff $k=p=n$.
    \item In the second coordinate, $t_2^l h_2 t_2^q h_2^{-1}$ commutes with $w_2$ iff $l=q=m$.
\end{itemize}
Thus, only the term corresponding to indices $(n,m)$ survives the expectation (assuming we filter out identity terms which we can by Case 1).
\[ \left\| (\mathbb{E}_{v_1} \otimes \mathbb{E}_{v_2})(y) - c_{t_1^n, t_2^m}^2 (t_1^n h_1 t_1^n h_1^{-1}) \otimes (t_2^m h_2 t_2^m h_2^{-1}) \right\| <(2\left\|a\right\|+\epsilon) \epsilon. \]
Since $c_{t_1^n, t_2^m} \neq 0$ and $\epsilon$ is arbitrary, we conclude $W := (t_1^n h_1 t_1^n h_1^{-1}) \otimes (t_2^m h_2 t_2^m h_2^{-1}) \in \mathcal{A}$.

Repeating the argument with $h_1^2, h_2^2$, we get $W' := (t_1^n h_1^2 t_1^n h_1^{-2}) \otimes (t_2^m h_2^2 t_2^m h_2^{-2}) \in \mathcal{A}$.
By standard algebraic manipulation (conjugating $W$ by $u^{-1} = h_1^{-1} \otimes h_2^{-1}$ and combining with $W'$), we isolate the element:
\[ (t_1^{2n} \otimes t_2^{2m}) \in \mathcal{A}. \]
Proceeding exactly as in the last steps of Theorem~\ref{thm:invfromsubgroup} (multiplying $x$ by a shifted conjugate and applying expectations), we eventually isolate the term $c_{t_1^n, t_2^m} (t_1^n \otimes t_2^m)$ in an $\epsilon$-neighborhood.
Thus $\lambda(g) \otimes \lambda(h) \in \mathcal{A}$.
\end{proof}
\begin{remark}
Our proof will show that $\Gamma=\Gamma_1\times\Gamma_2\times\ldots\times\Gamma_n$ has $C^*$-ISR property if every $\Gamma_i$ is a torsion-free non-amenable acylindrically hyperbolic.
\end{remark}
\section{Groups with the \texorpdfstring{$C^*$}{}-ISR property}
\label{sec:C*ISRrelations}

In this section, we show that for an infinite group $\Gamma$, if $\Gamma$ has the C$^*$-ISR property, then $\Gamma$ satisfies a certain dichotomy. In particular, either $\Gamma$ is a simple amenable group or $\Gamma$ is C$^*$-simple, i.e., $C^*_r(\Gamma)$ is a simple C$^*$-algebra. In particular, we prove Theorem~\ref{thm:conseISR}.

\begin{lem}\label{lem: obstruction to C*-ISR by non-codimension one ideals}
Assume that $C^*_r(\Gamma)$ has a non-zero proper closed two-sided ideal $I$ such that $C^*_r(\Gamma)/I\neq \mathbb{C}$. Then, $\Gamma$ does not have the C$^*$-ISR property.
\end{lem}
\begin{proof}
Set $A=I+\mathbb{C}$, which is a unital C$^*$-subalgebra in $C^*_r(\Gamma)$ and clearly $A$ is $\Gamma$-invariant. Assume that $\Gamma$ has the C$^*$-ISR property, then $A=C^*_r(\Lambda)$ for some normal subgroup $\Lambda\lhd \Gamma$.

First, observe that $\Lambda=\Gamma$. Indeed, first note that $I\neq (0)$ implies $\Lambda\neq \{e\}$. Then fix any $e\neq h\in \Lambda$, from $C^*_r(\Lambda)=I+\mathbb{C}$, we may write $\lambda(h)=a_h+c_h$ for some $a_h\in I$ and $c_h\in\mathbb{C}$. For any $g\in \Gamma$, we have $$C^*_r(\Lambda)\supset I\ni \lambda(g)a_h=\lambda(g)(\lambda(h)-c_h)=\lambda(gh)-c_h\lambda(g).$$
Note that $gh\neq g$, the above shows that $gh\in \Lambda$. Since $g$ is arbitrary, we deduce that $\Gamma=\Lambda$.

Then from $C^*_r(\Gamma)=I+\mathbb{C}$, we deduce that $C^*_r(\Gamma)/I=(I+\mathbb{C})/I\cong \mathbb{C}/{(I\cap \mathbb{C})}$.

Since $C^*_r(\Gamma)/I\neq \mathbb{C}$, we get that $I\cap \mathbb{C}\neq  (0)$, therefore, $I=C^*_r(\Gamma)$, which contradicts to the properness of the ideal $I$.
\end{proof}
    We now use the dynamics of $\Gamma$-boundary to show that a non-$C^*$-simple group with $C^*$-ISR property must be amenable. An action \( \Gamma \curvearrowright X \) is called a boundary action (in this case, we say that $X$ is a $\Gamma$-boundary) if, for every probability measure \( \nu \in \text{Prob}(X) \), \( \{\delta_x: x \in X\} \subset \overline{\Gamma \nu}^{\text{weak}^*} \).
The Furstenberg boundary of \( \Gamma \), denoted \( \partial_F\Gamma \), is the universal object in this category in the sense that for any other \( \Gamma \)-boundary \( Y \), there is a continuous $\Gamma$-equivariant surjective map from \( \partial_F\Gamma \) to \( Y \).

\begin{lem}\label{lem: non C*-simple groups vefify the existence of non-codimension one ideals}
Let $\Gamma$ be a countable infinite non-amenable group which is not C$^*$-simple. Then $C^*_r(\Gamma)$ contains a non-zero proper closed two-sided ideal $I$ such that $C^*_r(\Gamma)/I\neq \mathbb{C}$
\end{lem}
\begin{proof}
Since $\Gamma$ is not C$^*$-simple, there is a non-zero proper two-sided closed ideal $I\subset C^*_r(\Gamma)$. Assume that  $C^*_r(\Gamma)/I=\mathbb{C}$.

Denote by $\pi: C^*_r(\Gamma)\twoheadrightarrow C^*_r(\Gamma)/I=\mathbb{C}$ the *-homomorphism. Note that $\pi(\lambda(s))\overline{\pi(\lambda(s))}=\pi(\lambda(s)\lambda(s^{-1}))=\pi(\lambda(e))=1$. Thus $|\pi(\lambda(s))|^2=1$ for all $s\in \Gamma$.

Denote by $\partial_F\Gamma$ the Furstenberg boundary of $\Gamma$. We may extend the composition map $C^*_r(\Gamma)\overset{\pi}{\twoheadrightarrow}\mathbb{C}\overset{j}{\hookrightarrow}C(\partial_F\Gamma)$ by $\Gamma$-injectivity of $C(\partial_F\Gamma)$ (\cite[Theorem 3.12(1)]{KK}) to a $\Gamma$-equivariant u.c.p. map $\phi: C(\partial_F\Gamma)\rtimes_r\Gamma\rightarrow C(\partial_F\Gamma)$. Since $\phi|_{\mathbb{C}}=\text{Id}$, by rigidity (\cite[Theorem 3.12(2)]{KK}, $\phi|_{C(\partial_F\Gamma)}=\text{Id}$. Note that $\pi=\phi|_{C^*_r(\Gamma)}: C^*_r(\Gamma)\rightarrow \mathbb{C}$.

We now claim that $sx=x$ for all $s\in \Gamma$ and all $x\in \partial_F\Gamma$. If not, there exists some $x\in \partial_F\Gamma$ and some $e\neq s\in \Gamma$ such that $sx\neq x$. Choose $f\in C(\partial_F\Gamma)$ such that $f(x)=1$ and $f(s^{-1}x)=0$. Then the following calculation shows that
\begin{align*}
\pi(\lambda(s))=\phi(\lambda(s))
&=\phi(\lambda(s))(x)~~(\text{since $\phi(C^*_r(\Gamma))=\mathbb{C}$})\\
&=\phi(\lambda(s)f)(x)~~(\text{since $\phi|_{C(\partial_F\Gamma)}$=\text{Id}})\\
&=\phi(s.f\lambda(s))(x)\\
&=\phi(s.f)(x)\phi(\lambda(s))(x)\\
&=(s.f)(x)\phi(\lambda(s))(x)\\
&=f(s^{-1}x)\phi(\lambda(s))(x)\\
&=0.
\end{align*}
Thus we get a contradiction to $|\pi(\lambda(s))|^2=1$ as shown before.

By a well-known theorem of Furman \cite{Furman_amenableradical}, we see that $\Gamma=\text{Ker}(\Gamma\curvearrowright\partial_F\Gamma)=\text{Rad}(\Gamma)$, where $\text{Rad}(\Gamma)$ is the amenable radical of $\Gamma$, an amenable group. In particular, $\Gamma$ is amenable, a contradiction.
\end{proof}
For non-simple groups, it is easy to find non-trivial ideals with $C_r^*(\Gamma)/I\ne\mathbb{C}$. It is perhaps well-known to the experts; we record it for the sake of completion in the form of the following lemma.
\begin{lem}\label{lem: infinite amenable groups have non-codimension 1 representations iff G is non-simple}
Let $\Gamma$ be a countable infinite amenable group. Then $C^*_r(\Gamma)$ has a non-zero proper closed two-sided ideal $I$ such that $C^*_r(\Gamma)/I\neq \mathbb{C}$ if $\Gamma$ is non-simple.
\end{lem}
\begin{proof}
Since $\Gamma$ is non-simple, there exists a surjective group homomorphism $\phi: \Gamma\twoheadrightarrow \Lambda$, where $\Lambda$ is a non-trivial group. Then, it induces a surjective group homomorphism, still denoted by $\phi: C^*(\Gamma)\twoheadrightarrow C^*(\Lambda)$. Since $\Gamma$ is amenable, we deduce that $\Lambda$ is also amenable and hence we have C$^*$-isomorphisms $C^*(\Gamma)\cong C^*_r(\Gamma)$ and $C^*(\Lambda)\cong C^*_r(\Lambda)$. After composing with these two isomorphisms, we may assume that $\phi$ is a surjective *-homomorphism between reduced group C$^*$-algebras. Therefore $I:=\text{Ker}(\phi)$ is an ideal with the required property since $\Lambda$ is non-trivial.
\end{proof}
It was shown in \cite{kalantar2022invariant} and \cite{chifan2022invariant} that for groups considered there, every invariant $C^*$-subalgebra $\mathcal{A}\le C_r^*(\Gamma)$, which is the image of a conditional expectation, is of the form $C_r^*(N)$ for some normal subgroup $N\triangleleft\Gamma$. However, the ISR-property (that every invariant von Neumann subalgebra $\mathcal{M}\le L(\Gamma)$ is of the form $L(N)$ for some normal subgroup $N\triangleleft\Gamma$) does not imply the $C^*$-ISR property in general. The following example demonstrates that.
\begin{example}[ISR$\notimplies C^*$-ISR]\label{example: ISR does not imply C-ISR}
Recall that the amenable group $S_{\infty}$, the group consisting of all permutations on $\mathbb{N}$ with finite supports, is shown to have the ISR property in \cite{jiang2024example}. However, this group is not simple as it contains a normal subgroup of index two; $A_{\infty}$, the subgroup consisting of even permutations. Hence, $S_{\infty}$ does not have the C$^*$-ISR property by the combination of Lemma~\ref{lem: obstruction to C*-ISR by non-codimension one ideals} along with Lemma~\ref{lem: infinite amenable groups have non-codimension 1 representations iff G is non-simple}. More explicitly, letting $I=\text{Ker}(\pi)$, where $\pi: C^*(S_{\infty})\rightarrow C^*(\mathbb{Z}/2\mathbb{Z})$, $\mathcal{A}_I=I+\mathbb{C}$ is an invariant $C^*$-subalgebra which does not come from a subgroup.
\end{example}
Combining all of these above, we obtain the following dichotomy for groups with $C^*$-ISR property. This is Theorem~\ref{thm:conseISR} from the introduction.
\begin{theorem}
Let $\Gamma$ be a countable infinite group with the C$^*$-ISR property. Then $\Gamma$ is either a simple amenable group or is a C$^*$-simple group. In particular, $\Gamma$ is an i.c.c. group.
\end{theorem}
\begin{proof}
The first part follows by combining Lemma \ref{lem: obstruction to C*-ISR by non-codimension one ideals}, Lemma \ref{lem: non C*-simple groups vefify the existence of non-codimension one ideals} and Lemma \ref{lem: infinite amenable groups have non-codimension 1 representations iff G is non-simple}.

For the last part, note that the finite radical $\Gamma_{\text{fin}}$, i.e., the normal subgroup consisting of all elements in $\Gamma$ with finite conjugacy classes, are amenable. Indeed, this follows since for any finitely many elements $s_1,\ldots, s_n\subset \Gamma_{\text{fin}}$, we have that $\langle s_1,\ldots, s_n\rangle$ is virtually abelian. To see  this, observe that $[\Gamma: C_\Gamma(\langle s_1,\ldots, s_n)]=[\Gamma: \cap_{i=1}^nC_\Gamma(s_i)]<\infty$ since $[\Gamma: C_\Gamma(s_i)]<\infty$ for each $1\leq i\leq n$. Write $C(\langle s_1,\ldots, s_n\rangle)$ for the center of $\langle s_1,\ldots, s_n\rangle$. Observe that $\langle s_1,\ldots, s_n\rangle/{C(\langle s_1,\ldots, s_n\rangle)}$ embeds into $\Gamma/{C_\Gamma(\langle s_1,\ldots, s_n\rangle)}$ and hence is a finite group.

It is well-known that C$^*$-simple groups have trivial amenable radical, we deduce that $\Gamma_{\text{fin}}$ is trivial, i.e. $\Gamma$ is i.c.c. if $\Gamma$ is a C$^*$-simple group.

If $\Gamma$ is an infinite simple group, then $\Gamma$ is i.c.c. Indeed, let $e\neq s\in \Gamma$ be an element with finite conjugacy class, equivalently, $[\Gamma: C_\Gamma(s)]<\infty$. Then there is a finite index normal subgroup $\Lambda\lhd \Gamma$ such that $\Lambda\subset C_\Gamma(s)$. Since $\Gamma$ is simple, we deduce that either $\Lambda=\{e\}$ or $\Lambda=\Gamma$. If $\Lambda=\{e\}$, then $\Gamma$ is finite, a contradiction. If $\Lambda=\Gamma$, then $\Gamma=C_\Gamma(s)$ and hence $s$ lies in the center of $\Gamma$. Combining this with the simplicity of $\Gamma$, we deduce that $\Gamma$  is abelian. But an infinite abelian group is not simple, a contradiction.
\end{proof}
\begin{remark}
An unital $C^*$-algebra $\mathcal{A}$ is called just-infinite if it is infinite-dimensional,
and for each non-zero closed two-sided ideal $I\triangleleft \mathcal{A}$, the quotient $\mathcal{A}/I$ is finite-dimensional (see \cite[Definition~3.1]{GMR}). To show that an infinite group $\Gamma$ with the $C^*$-ISR property is either $C^*$-simple or amenable, we can also apply \cite[Proposition 6.1]{GMR} directly since Lemma \ref{lem: obstruction to C*-ISR by non-codimension one ideals} shows that $C^*_r(\Gamma)$ is a just-infinite $C^*$-algebra in the sense of \cite{GMR}.
\end{remark}
\begin{remark}
To show an infinite group with C$^*$-ISR property is i.c.c., we can also repeat the argument in \cite[Proposition 3.1]{amrutam2023invariant} with necessary modification by using Lemma \ref{lem: infinite amenable groups have non-codimension 1 representations iff G is non-simple}.
\end{remark}
Let $\text{Ker}(\epsilon)$ denote the augmentation ideal defined as the kernel of the trivial representation $\epsilon: C^*(\Gamma)\rightarrow \mathbb{C}$; equivalently, it is the closed ideal in $C^*(\Gamma)$ generated by $\{\lambda(g)-\lambda(e):~g\in\Gamma\}$.
We now observe whether infinite amenable simple groups satisfy the $C^*$-ISR property. Under the assumption that $\mathbb{C}\Gamma$ has a unique $C^*$-completion, it turns out that $\text{Ker}(\epsilon)$ is the unique non-trivial ideal. We refer the reader to \cite{GMR} for a detailed discussion on the uniqueness of C$^*$-completions. The group algebra $\mathbb{C}\Gamma$ is said to have a unique C$^*$-completion if there is exactly one C$^*$-norm on $\mathbb{C}\Gamma$. This property is intimately connected to the ideal structure of the group C$^*$-algebra. For instance, it was shown in \cite[Proposition~6.7]{GMR} that every locally finite group has the unique C$^*$-completion property. Moreover, the study of just-infinite C$^*$-algebras in \cite{GMR} highlights that for many groups, the ideal structure is rigid (finite codimension ideals), a phenomenon that parallels the consequences of the C$^*$-ISR property we investigate here.
\begin{prop}
Let $\Gamma$ be a countable infinite amenable group. If $\mathbb{C}\Gamma$ has a unique C$^*$-completion and the only non-zero proper closed two-sided ideal in $\mathbb{C}\Gamma$ is $\text{Ker}(\epsilon)\cap \mathbb{C}\Gamma$, i.e., the augmentation ideal in $\mathbb{C}\Gamma$, then the only non-zero proper closed two-sided ideal in $C^*(\Gamma)=C_r^*(\Gamma)$ is $\Phi(\text{Ker}(\epsilon))$.
\end{prop}
\begin{proof}
Assume that $(0)\neq I\subset C^*_r(\Gamma)$ is a closed two-sided ideal, then $\Phi^{-1}(I)\cap \mathbb{C}\Gamma\neq (0)$ by the assumption that $\mathbb{C}\Gamma$ has a unique C$^*$-completion and $I\neq (0)$. Hence $\Phi^{-1}(I)\cap \mathbb{C}\Gamma=\text{Ker}(\epsilon)\cap \mathbb{C}\Gamma$ or $\mathbb{C}\Gamma$. Hence $\text{Ker}(\epsilon)\subseteq \Phi^{-1}(I)$ since the completion of $\text{Ker}(\epsilon)\cap \mathbb{C}\Gamma$ in $C^*_r(\Gamma)$ is $\text{Ker}(\epsilon)$. Since $\text{Ker}(\epsilon)$ has codimension one in $C^*(\Gamma)$, we deduce that $I=\Phi(\text{Ker}(\epsilon))$ or $C^*_r(\Gamma)$.
\end{proof}
We do not know if infinite simple amenable groups satisfy the $C^*$-ISR property. However, under the assumption that they do, we can shed some light on the ideal structure of $C_r^*(\Gamma)$. In particular, we now prove Theorem~\ref{thm:consequenceofISR}.
\begin{theorem}
Let $\Gamma$ be a countable infinite amenable group. Let $\Phi: C^*(\Gamma)\twoheadrightarrow C^*_r(\Gamma)$ be the canonical *-isomorphism such that $\Phi(\lambda(s))=\lambda(s)$ for all $s\in\Gamma$. Assuming that $\Gamma$ has the C$^*$-ISR property, the following statements hold true.
\begin{itemize}
\item[(1)] The only closed two-sided ideals in $C^*_r(\Gamma)$ are $(0)$, $\Phi(\text{Ker}(\epsilon))$, or $C^*_r(\Gamma)$, where
\item[(2)] $\mathbb{C}\Gamma$ has a unique C$^*$-completion.
\end{itemize}
\end{theorem}
\begin{proof}
(1) Let $I\subset C_r^*(\Gamma)$ be a non-zero closed two-sided proper ideal. Since $\Gamma$ has the C$^*$-ISR property, we deduce from  Lemma \ref{lem: infinite amenable groups have non-codimension 1 representations iff G is non-simple} and Lemma \ref{lem: obstruction to C*-ISR by non-codimension one ideals} that $\Gamma$ is simple and $\Phi^{-1}(I)$ has codimension one in $C^*(\Gamma)$. In other words, $C^*(\Gamma)/{\Phi^{-1}(I)}=\mathbb{C}$. Note that the surjective *-homomorphism $\pi: C^*(\Gamma)\twoheadrightarrow C^*(\Gamma)/{\Phi^{-1}(I)}$ gives rise to a group homomorphism $\pi|_{\Gamma}: \Gamma\rightarrow \mathcal{U}(C^*(\Gamma)/{\Phi^{-1}(I)})=\mathcal{U}(\mathbb{C})=\mathbb{T}$.
Moreover, note that $\pi|_{\Gamma}$ factorizes through $\Gamma_{ab}:= \Gamma/{[\Gamma, \Gamma]}$ since $\mathbb{T}$ is abelian.

Since  $\Gamma$ is simple and infinite, we get that $\Gamma_{ab}=\{e\}$. Hence $\pi|_{\Gamma}(\lambda(t))=1$ for all $t\in \Gamma$. This shows that $\pi$ factorizes through the augmentation ideal $\text{Ker}(\epsilon)$ and thus $\text{Ker}(\epsilon)\subseteq \Phi^{-1}(I)\subset C^*(\Gamma)$. Since both $\text{Ker}(\epsilon)$ and $\Phi^{-1}(I)$ have codimension one in $C^*(\Gamma)$, we deduce that $\text{Ker}(\epsilon)=\Phi^{-1}(I)$, i.e., $I=\Phi(\text{Ker}(\epsilon))$.

(2) By \cite[Lemma 2.2]{AK_pacific}, $\mathbb{C}\Gamma$ has a unique C$^*$-completion iff $\mathbb{C}\Gamma$ satisfies that for any non-zero closed two-sided ideal $J$ in $C^*(\Gamma)$, we have that $J\cap \mathbb{C}\Gamma\neq (0)$. Since it is clear that $\text{Ker}(\epsilon)\cap\mathbb{C}\Gamma\neq (0)$, the proof is complete by part (1). \end{proof}
\begin{remark}
There are examples of infinite simple amenable groups satisfying the two conditions in the above proposition. Indeed, by  \cite[Proposition 6.7]{GMR}, if $\Gamma$ is a locally finite group (i.e., if every finitely generated subgroup is finite), then $\mathbb{C}\Gamma$ has a unique C$^*$-completion. Furthermore,  there are locally finite groups $\Gamma$ such that the only non-zero proper ideal in $\mathbb{C}\Gamma$ are just the augmentation ideals, such as Hall's universal groups, see \cite{BHPS} and \cite{LP}. So, these groups satisfy these two conditions, but we do not know whether these groups satisfy the C$^*$-ISR property.
\end{remark}

\begin{question}
Does the C$^*$-ISR property imply the ISR property?
\end{question}
\begin{question}
Let $\Gamma$ be an infinite group. If $\Gamma$ has the C$^*$-ISR property, then is $\Gamma$  C$^*$-simple? Does the converse hold?
\end{question}
\newpage
\bibliographystyle{amsalpha}
\bibliography{Inv}
\end{document}